\documentclass[12pt,a4paper]{article}
\usepackage{amsmath,amssymb,amsthm,graphicx,color}
\usepackage[left=1in,right=1in,top=1in,bottom=1in]{geometry}
\usepackage{cite}
\usepackage[british]{babel}

\newtheorem{theo}{Theorem}[section]
\newtheorem{lm}{Lemma}[section]

\newtheorem{cor}{Corollary}[section]
\newtheorem{rmk}{Remark}[section]

\newtheorem{prop}{Proposition}[section]

 \allowdisplaybreaks

\numberwithin{equation}{section}

\newcommand{\R}{{\mathbb R}}

\newcommand{\ZP}{{\mathbb Z}_+}
\newcommand{\N}{{\mathbb N}}
\newcommand{\Exp}{{\mathbb E}}
\renewcommand{\Pr}{{\mathbb P}}
\newcommand{\1}{{\mathbf 1}}
\newcommand{\eps}{\varepsilon}

\newcommand{\toas}{{\stackrel{a.s.}{~\longrightarrow~}}}
\newcommand{\tod}{{\stackrel{d}{~\longrightarrow~}}}
\newcommand{\eqd}{{\stackrel{d}{~=~}}}

\newcommand{\re}{{\mathrm{e}}}
\newcommand{\XX}{{\cal X}}
\newcommand{\ud}{{\mathrm d}}
\newcommand{\FF}{{\mathcal F}}
\newcommand{\as}{\ \textrm{a.s.}}

\newcommand{\card}{\mathop{\#}}
\newcommand{\essinf}{{\rm ess} \inf}
\newcommand{\esssup}{{\rm ess} \sup}
\newcommand{\ord}{\mathop{\mathrm{ord}}}
\newcommand{\hull}{\mathop{\mathrm{hull}}}

\title{Convergence in a multidimensional randomized {K}eynesian beauty contest}

\author{Michael Grinfeld\footnote{Department of Mathematics and Statistics, University of Strathclyde, 26 Richmond Street, Glasgow G1 1XH, UK.}
\and
Stanislav Volkov\footnote{Centre for Mathematical Sciences, Lund University, Box 118 SE-22100, Lund, Sweden, and Department of Mathematics, University of Bristol, University Walk, Bristol BS8 1TW, UK.}
 \and Andrew R.\ Wade\footnote{Department of Mathematical Sciences, Durham University, South Road, Durham DH1 3LE, UK.}
  }

\begin{document}

\maketitle

\begin{abstract}
We study the asymptotics of a Markovian system of $N \geq 3$ particles in $[0,1]^d$ in which, at each step in discrete time,
the particle farthest from the current centre of mass is removed and replaced by an independent $U [0,1]^d$ random particle.
We show that the
 limiting configuration contains $N-1$ coincident particles at a random location $\xi_N \in [0,1]^d$.
A key tool
in the analysis is a Lyapunov function based on the squared radius of gyration (sum of squared distances) of the points.
For $d=1$ we give additional results on the distribution of the limit $\xi_N$,
showing, among  other things, that it gives positive probability to any nonempty interval subset of $[0,1]$,
and giving a reasonably explicit description in the smallest nontrivial case, $N=3$.
\end{abstract}

\smallskip
\noindent
{\em Keywords:} Keynesian beauty contest;  radius of gyration; rank-driven process; sum of squared distances. \/

\noindent
{\em AMS 2010 Subject Classifications:} 60J05 (Primary) 60D05, 60F15,  60K35, 82C22, 91A15 (Secondary)

\section{Introduction, model, and results}

In a {\em Keynesian beauty contest}, $N$ players each guess a number, the winner being the player whose
guess is closest to the mean of all the $N$ guesses; the name marks
Keynes's discussion of
 ``those newspaper competitions in which the competitors have to pick out the six prettiest faces from a hundred photographs, the prize being awarded to the competitor whose choice most nearly corresponds to the average preferences of the competitors as a whole'' \cite[Ch.\ 12, \S V]{keynes}.
 Moulin \cite[p.\ 72]{moulin} formalized a  version of the game played on a real interval, the
 ``$p$-beauty contest'', in which the target is $p$ ($p>0$) times the mean value. See e.g.\ \cite{degr} and references therein for
 some recent work on game-theoretic aspects of such ``contests'' in economics.

 In this paper we study a stochastic process
 based on an iterated  version of  the game, in which players  {\em randomly} choose a value in $[0,1]$, and at each
 step the worst performer (that is, the player whose guess is farthest from the mean)
  is replaced by a new player; each player's guess is fixed as soon as they enter the game, so a single
 new random value enters the system at each step.
 Analysis of this model was posed as an open problem in \cite[p.\ 390]{gkw}.
 The natural setting for our techniques is in fact a generalization in which
 the values live in $[0,1]^d$ and the target is the barycentre (centre of mass) of the values. We now formally describe the model and state our main results.

Let $d \in \N := \{1,2,\ldots\}$.
We use the notation $\XX_n = ( x_1, x_2, \ldots, x_n )$ for a vector of $n$ points $x_i \in \R^d$.
We write $\mu_n (\XX_n)
:= n^{-1} \sum_{i=1}^n x_i$
for the barycentre of $\XX_n$, and $\| \, \cdot \, \|$ for the Euclidean norm on $\R^d$.
Let
$\ord ( \XX_n ) = ( x_{(1)} , x_{(2)} , \ldots, x_{(n)})$
denote the {\em barycentric order statistics} of $x_1, \ldots, x_n$, so that
$ \| x_{(1)} - \mu_n (\XX_n ) \| \leq \| x_{(2)} - \mu_n (\XX_n) \| \leq \cdots \leq \| x_{(n)} - \mu_n (\XX_n) \|$;
any ties are broken randomly.
We call $\XX^*_n :=  
x_{(n)}$ the {\em extreme} point of $\XX_n$, a point of $x_1, \ldots, x_n$
  farthest from the barycentre. We define the {\em core} of $\XX_n$ as
$\XX_n'
:= (x_{(1)}, \ldots, x_{(n-1)})$, the vector of $x_1, \ldots, x_n$ with the extreme point removed.

The Markovian model that we study is defined as follows.
Fix $N \geq 3$. Start with $X_1(0), \ldots, X_N(0)$, distinct points in $[0,1]^d$, and write $\XX_N (0) := ( X_{(1)} (0), \ldots, X_{(N)} (0) )$
for the corresponding ordered vector. One possibility is to start with a {\em uniform random} initial configuration, by taking
$X_1(0), \ldots, X_N (0)$ to be independent $U[0,1]^d$ random variables; here and elsewhere $U [0,1]^d$ denotes the uniform
distribution on $[0,1]^d$. In this uniform random initialization,
all $N$ points are indeed distinct with probability 1.
Given $\XX_N(t)$, replace $\XX^*_N(t) = X_{(N)} (t)$ by an independent $U[0,1]^d$ random variable $U_{t+1}$, 
so that $\XX_N (t+1) = \ord ( X_{(1)} (t), \ldots, X_{(N-1)} (t), U_{t+1} )$.

The  interesting case is  when $N \geq 3$: the case $N=1$ is trivial, and the case $N=2$ is also
uninteresting since at each step either point is replaced with probability $1/2$ by a $U[0,1]$ variable, so that,
regardless of the initial configuration, after a finite number of steps we will have two independent $U[0,1]$ points.
Our main result, Theorem \ref{thm1}, shows that for $N \geq 3$ all but the most extreme point of the configuration converge
to a common limit.
 
\begin{theo}
\label{thm1}
Let $d\in \N$ and $N \geq 3$. Let $\XX_N(0)$ consist of $N$ distinct points in $[0,1]^d$.
There exists a random $\xi_N := \xi_N (\XX_N(0)) \in [0,1]^d$ such that
\begin{equation}
\label{thm1a}
 \XX'_N(t) \toas ( \xi_N, \xi_N, \ldots, \xi_N) , \textrm{~~and~~} \XX_N^*(t) - U_t \toas 0 ,\end{equation}
 as $t \to \infty$.
In particular, for $U \sim U  [0,1]^d  $, as $t \to \infty$,
\[ \XX_N (t) \tod ( \xi_N, \xi_N, \ldots, \xi_N, U ) .\]
\end{theo}

\begin{rmk}
\label{rem0}
Under the conditions of Theorem \ref{thm1}, despite the fact that $\XX_N^* (t) - U_t \to 0$ a.s.,  we will see below that $\XX_N^* (t) \neq U_t$ infinitely often a.s.
\end{rmk}

  Theorem \ref{thm1} is proved in Section \ref{sec:proof}. Then,  Section \ref{sec:1d} is devoted to the one-dimensional case, where we obtain various
  additional results on the limit $\xi_N$. Finally, the Appendices, Sections \ref{appendix}
and \ref{appendix2}, collect some results on uniform spacings and continuity of
distributional fixed-points that we use in parts of the analysis
  in Section \ref{sec:1d}.

\section{Proof of convergence}
\label{sec:proof}

Intuitively, the evolution of the  process is as follows. If, on replacement of the extreme point, the new point is the next extreme point
(measured with respect to the new centre of mass), then the core is unchanged. However, if the new point is not extreme, it typically penetrates
the core significantly, while a more extreme point is thrown out of the core, reducing the size of the core in some sense (we give a precise
statement below). Tracking the evolution of the core, by following its centre of mass, one sees increasingly long periods of inactivity, since
as the size of the core decreases  changes occur less often, and moreover the magnitude of the changes decreases in step with the size of the core.
The dynamics are nontrivial, but bear some resemblance to random walks with decreasing steps (see e.g.\ \cite{kr,erdos} and references therein)
as well as processes with reinforcement such as the P\'olya urn (see e.g.\ \cite{pemantle} for a survey).

Our analysis will rest on a `Lyapunov function' for the process, that is, a function of the configuration
that possesses pertinent asymptotic properties. One may initially hope, for example, that the diameter of the point set $\XX_N(t)$
would decrease over time, but this cannot be the case because the newly added point can be anywhere in $[0,1]^d$. What then about the
diameter of $\XX_N'(t)$, for which the extreme point is ignored? We will show later in this section
 that this quantity is in fact well behaved,
but we have to argue somewhat indirectly: the diameter of $\XX_N'(t)$ can increase (at least for $N$ big enough; see Remark \ref{rmk1} below).
However, there {\em is} a monotone decreasing function associated with the process, based on the sum of squared distances of a configuration,
which we will use as our Lyapunov function.

For $n \in \N$ and $\XX_n = ( x_1, x_2, \ldots, x_n )  \in \R^{dn}$, write
\begin{equation}
\label{Gdef}
G_n ( \XX_n ) := G_n (x_1, \ldots, x_n ) := n^{-1} \sum_{i=1}^n \sum_{j=1}^{i-1} \| x_i - x_j \|^2 = \sum_{i=1}^n \| x_i - \mu_n ( \XX_n ) \|^2 ;
\end{equation}
a detailed proof of the (elementary) final equality in (\ref{Gdef}) may be found on pp.~95--96 of \cite{hughes}, for example.
 We remark that $\frac{1}{n}G_n$ is  the squared {\em radius of gyration} of $x_1, \ldots, x_n$: see e.g.\ \cite{hughes}, p.\ 95.
Note also that calculus verifies the useful variational formula
\begin{equation}
\label{variational}
G_n (x_1, \ldots, x_n ) = \inf_{y \in \R^d} \sum_{i=1}^n \| x_i - y \|^2 .\end{equation}

For $n \geq 2$, define
\[ F_n ( \XX_n ) :=  F_n ( x_1, \ldots, x_n ) := G_{n-1} ( \XX_n'   ) = G_{n-1} (x_{(1)}, \ldots, x_{(n-1)}) .\]

\begin{lm}
\label{lem1}
Let $n \geq 2$ and $\XX_n = ( x_1, x_2, \ldots, x_n )  \in \R^{dn}$. Then for any $x \in \R^d$,
\[ F_n ( x_{(1)}, \ldots, x_{(n-1)}, x ) \leq F_n ( \XX_n ) . \]
\end{lm}
\begin{proof}
 For ease of notation, we write simply $(x_1, \ldots, x_n)$ for $(x_{(1)}, \ldots, x_{(n)})$,
 i.e., we relabel so that $x_j$ is the $j$th closest point to $\mu_n ( \XX_n )$. Then $\XX_n^* = x_n$,
 $\XX'_n = (x_1,\ldots, x_{n-1})$, and
\begin{equation}
\label{fold}
F_{\rm old} := F_n (\XX_n )  = G_{n-1} ( x_1, \ldots, x_{n-1} ) = \sum_{i=1}^{n-1} \| x_i -  \mu'_{\rm old} \|^2 ,
 \end{equation}
where
$\mu'_{\rm old} := \mu_{n-1} (\XX_n' )$.
We compare $F_{\rm old}$ to $F_n$ evaluated on the set of points obtained
by removing  $x_n$ and replacing it with some $x \in \R^d$.

Write $y := \{ x_1, \ldots, x_{n-1} , x\}^*$ for the new extreme point. Then
\begin{equation}
\label{fnew}
 F_{\rm new} := F_n ( x_{1}, \ldots, x_{n-1}, x )   = \sum_{i=1}^{n-1} \| x_i - \mu'_{\rm new} \|^2
+ \| x-  \mu'_{\rm new} \|^2 - \| y-  \mu'_{\rm new} \|^2 ,
\end{equation}
where
\begin{equation}
\label{munew}
 \mu'_{\rm new} := \frac{1}{n-1} \left( \sum_{i=1}^{n-1}x_i + x - y \right) =  \mu'_{\rm old} + \frac{x-y}{n-1} .
\end{equation}
Denote $\mu_{\rm new} := \mu_n (x_1, \ldots, x_{n-1}, x)$, so
\begin{equation}
\label{eq3}
 \mu'_{\rm new}=\frac{n \mu_{\rm new}}{n-1}-\frac y{n-1}.
\end{equation}
From (\ref{fold}), (\ref{fnew}), and (\ref{munew}),  we obtain
\begin{align}
\label{eq22}
 F_{\rm new}   - F_{\rm old}
  & = \sum_{i=1}^{n-1} \left( \|x_i -  \mu'_{\rm new} \|^2 - \|x_i -  \mu'_{\rm old} \|^2 \right) + \|x - \mu'_{\rm new} \|^2 - \|y -  \mu'_{\rm new} \|^2 .
\end{align}
For the sum on the right-hand side of (\ref{eq22}), we have that
\begin{align*}
\sum_{i=1}^{n-1} \left( \|x_i -  \mu'_{\rm new} \|^2 - \|x_i -  \mu'_{\rm old} \|^2 \right) & = \sum_{i=1}^{n-1}
\left( 2 x_i \cdot ( \mu'_{\rm old} -  \mu'_{\rm new} )    + \|  \mu'_{\rm new} \|^2 - \| \mu'_{\rm old} \|^2 \right) \\
&  = (n-1) \left( 2  \mu'_{\rm old} \cdot ( \mu'_{\rm old} -  \mu'_{\rm new} ) + \|  \mu'_{\rm new} \|^2 - \| \mu'_{\rm old} \|^2 \right) \\
& = (n-1) \left(  \|  \mu'_{\rm old} \|^2  - 2 ( \mu'_{\rm old} \cdot  \mu'_{\rm new} ) + \|  \mu'_{\rm new} \|^2  \right).
\end{align*}
Simplifying this last expression and substituting back into (\ref{eq22}) gives $F_{\rm new} - F_{\rm old}
= (n-1) \|  \mu'_{\rm old} -  \mu'_{\rm new} \|^2 +  \|x -  \mu'_{\rm new}\|^2 -
\|y -  \mu'_{\rm new} \|^2$. Thus, using (\ref{munew}) and then (\ref{eq3}),
\begin{align*}
F_{\rm new} - F_{\rm old}
 &
= \frac{\|x-y\|^2}{n-1}
 +\|x\|^2-\|y\|^2 -  2  \mu'_{\rm new} \cdot (x-y)
\\
&= \frac{\|x\|^2+\|y\|^2-2x\cdot y}{n-1}
 +\|x\|^2-\|y\|^2 -  2\left(\frac{n\mu_{\rm new}}{n-1}-\frac y{n-1}\right) \cdot (x-y). \end{align*}
 Hence we conclude that
 \begin{align}
 \label{Fchange}
 F_{\rm new} - F_{\rm old}
& = \frac{n}{n-1} \left( \| x \|^2 - \|y\|^2 - 2 \mu_{\rm new} \cdot (x-y) \right)\nonumber \\
&=\frac n{n-1} \left(\|x-\mu_{\rm new}\|^2-\|y-\mu_{\rm new}\|^2\right) \leq 0,
\end{align}
since $y$ is, by definition, the farthest point from $\mu_{\rm new}$.
\end{proof}

Consider $F(t) := F_N( \XX_N (t) )$. Lemma \ref{lem1} has the following immediate consequence.

\begin{cor}
\label{cor1}
Let $N \geq 2$. Then $F (t+1) \leq F (t)$.
\end{cor}

Corollary \ref{cor1} shows that our Lyapunov function $F(t)$ is nonincreasing; later we show that $F(t) \to 0$ a.s.\ (see Lemma \ref{lem4} below).
First, we need to relate $F(t)$ to the {\em diameter} of the point set $\XX'_N(t)$.
For $n \geq 2$ and $x_1, \ldots, x_n \in \R^d$, write
\[ D_n ( x_1, \ldots, x_n ) := \max_{1 \leq i,j \leq n} \| x_i - x_j \|. \]

\begin{lm}
\label{lem2}
Let $n \geq 2$ and $x_1, \ldots, x_n \in \R^d$. Then
\[ \frac{1}{2} D_n ( x_1, \ldots, x_n )^2 \leq G_n (x_1, \ldots, x_n ) \leq \frac{1}{2}  (n-1) D_n ( x_1, \ldots, x_n )^2 .\]
\end{lm}

\begin{rmk}
The lower bound in Lemma \ref{lem2} is sharp, and is attained by collinear configurations with two diametrically opposed points $x_i, x_j$ and all the other $n-2$ points
at the midpoint $\mu_2 (x_i,x_j) = \mu_n (x_1, \ldots, x_n)$. The upper bound in Lemma \ref{lem2} is not, in general, sharp; determining the sharp
upper bound is a nontrivial problem. The bound $G_n (x_1, \ldots, x_n ) \leq \frac{n}{2} \left(\frac{d}{d+1} \right)  D_n ( x_1, \ldots, x_n )^2$
\cite{wit} is also not always sharp. Witsenhausen \cite{wit}
 conjectured that the maximum is attained if and only if the points are
distributed as evenly as possible among the vertices of a regular $d$-dimensional
simplex of edge-length $D_n(x_1, \ldots, x_n)$; this conjecture was proved relatively recently \cite{pill,bm}.
\end{rmk}

\begin{proof}[Proof of Lemma \ref{lem2}.]
Fix $x_1, \ldots, x_n \in \R^d$. For ease of notation, write $\mu = \mu_n (x_1, \ldots, x_n)$.
First we prove the lower bound. For $n \geq 2$, using the second form of $G_n$ in (\ref{Gdef}),
\begin{align*} G_n (x_1, \ldots, x_n ) & = \sum_{i=1}^n \| x_i - \mu \|^2 \geq \| x_i - \mu \|^2 + \| x_j - \mu \|^2 ,\end{align*}
where $(x_i,x_j)$ is a diameter, i.e., $D_n(x_1,\ldots,x_n ) = \| x_i - x_j \|$.
By the $n=2$ case of (\ref{variational}),
\[  \| x_i - \mu \|^2 + \| x_j - \mu \|^2 \geq 2 \| x_i - \mu_2 (x_i, x_j ) \|^2 = \frac{1}{2} \| x_i - x_j \|^2 .\]
This gives the lower bound. For the upper bound,  from the first form of $G_n$ in (\ref{Gdef}),
\[ G_n (x_1, \ldots, x_n ) \leq \frac{1}{n} \sum_{i=1}^n (i-1) D_n (x_1, \ldots, x_n)^2 ,\]
by the definition of $D_n$, which yields the result.
\end{proof}

Let $D(t) := D_{N-1} ( \XX_N ' (t))$.

\begin{rmk}
\label{rmk1}
By Lemma \ref{lem2} (or (\ref{Gdef})), $G_2 ( \XX_3'(t)) = \frac{1}{2} D_2 ( \XX_3' (t))^2$,
so when $N=3$, Lemma \ref{lem1} implies that $D (t+1) \leq D (t)$ a.s.\ as well.
If $d=1$, it can  be shown that $D (t)$ is nonincreasing also when $N=4$. In general, however, $D (t)$ can increase.
\end{rmk}

Let $\FF_t := \sigma ( \XX_N(0), \XX_N(1), \ldots, \XX_N(t) )$, the $\sigma$-algebra generated by the process up to time $t$.
 Let $B ( x; r)$
denote the closed Euclidean $d$-ball with centre $x \in \R^d$ and radius $r >0$. Define the events
\[ A_{t+1} := \{ U_{t+1} \in B ( \mu_{N-1} ( \XX_N' (t) )  ;  3 D (t) ) \}, ~~
A'_{t+1} := \{ U_{t+1} \in B ( \mu_{N-1} ( \XX_N' (t) )  ; D(t) /4 ) \}. \]

\begin{lm}
\label{lem3}
There is an absolute constant $\gamma >0$ for which, for all $N \geq 3$ and all $t$,
\begin{equation}
\label{events}
A'_{t+1} \subseteq \{ F(t+1) - F(t) \leq - \gamma N^{-1} F(t) \} \subseteq \{ F(t+1) - F(t) < 0 \} \subseteq   A_{t+1}   .\end{equation}
Moreover, there exist constants $c >0$ and $C<\infty$, depending only on $d$, for which, for all $N \geq 3$ and all $t$, a.s.,
\begin{align}
\label{eq5}
 \Pr \left[ F (t+1) - F (t) \leq - \gamma N^{-1} F (t) \mid \FF_t \right] & \geq
 \Pr [ A'_{t+1} \mid \FF_t ] \geq  c N^{-d/2} ( F(t) )^{d/2} ; \\
\label{eq6}
 \Pr \left[ F(t+1) - F(t) < 0 \mid \FF_t \right] & \leq
 \Pr [ A_{t+1} \mid \FF_t ] \leq C  ( F(t) )^{d/2} .\end{align}
\end{lm}
\begin{proof}
For simplicity we write $X_1, \ldots, X_{N-1}$ instead of $X_{(1)} (t), \ldots, X_{(N-1)} (t)$
 and  $D$ instead of $D(t) = D_{N-1} (X_1, \ldots, X_{N-1})$.
By definition of $D$, there exists some $i \in \{1,\ldots, N-1\}$
such that $\|  \mu'_{\rm old} - X_i \| \geq D/2$, where
$ \mu'_{\rm old} = \mu_{N-1} ( X_{1}, \ldots, X_{N-1} )$.
Given $\FF_t$, the event $A'_{t+1}$, that the new point $U := U_{t+1}$ falls in $B ( \mu'_{\rm old} ; D/4)$, has probability
bounded below by $\theta_d D^d$, where $\theta_d>0$ depends only on $d$.
Let $\mu_{\rm new} := \mu_N ( X_1, \ldots, X_{N-1}, U )$. Suppose that $A'_{t+1}$ occurs.
Then, 
\begin{equation}
\label{eq0} \| \mu_{\rm new} -  \mu'_{\rm old} \| = \frac{1}{N} \| U -  \mu'_{\rm old} \| \leq \frac{D}{4N} \leq \frac{D}{12} ,\end{equation}
since $N \geq 3$.
Hence, by (\ref{eq0}) and the triangle inequality,
\begin{equation}
\label{eq1}
 \| U - \mu_{\rm new} \| \leq \| U -  \mu'_{\rm old} \| + \| \mu_{\rm new} - \mu'_{\rm old} \|
\leq \frac{D}{4} + \frac{D}{12} = \frac{4D}{12} .\end{equation}
On the other hand, by another application of the triangle inequality and (\ref{eq0}),
\[ \| \mu_{\rm new} - X_i \| \geq \| \mu'_{\rm old} - X_i \| - \| \mu_{\rm new} - \mu'_{\rm old} \|
\geq \frac{D}{2} - \frac{D}{12} = \frac{5D}{12} .\]
Then, by definition, the extreme point $Y := \{ X_1, \ldots, X_{N-1}, U \}^*$ satisfies
\begin{equation}
\label{eq98}
 \| Y - \mu_{\rm new} \|  \geq \| \mu_{\rm new} - X_i \| \geq \frac{5D}{12} .\end{equation}
Hence from the $x=U$ case of (\ref{Fchange}) with the bounds (\ref{eq1}) and (\ref{eq98}), we conclude that
\begin{equation}
 \label{eq2}
  F (t+1) - F(t)   \leq  \frac{N}{N-1} \left( \left( \frac{4D}{12} \right)^2 - \left( \frac{5D}{12} \right)^2 \right) \1 (A'_{t+1} ) \leq - 
  \frac{9}{144} D^2 \1 (A'_{t+1} )  ,\end{equation}
for  
all $N \geq 3$; the first inclusion in (\ref{events})  follows (with $\gamma = 9/72$)
from (\ref{eq2}) together with the fact that,
by the second inequality in Lemma \ref{lem2}, $D^2 \geq 2 N^{-1} F(t)$. This in turn implies (\ref{eq5}), using the fact that $\Pr [ A'_{t+1} \mid \FF_t ] \geq \theta_d D^d$.

Next we consider the event $A_{t+1}$. Using the same notation as above, we have that
\[ \| \mu_{\rm new} - U \| \geq \| \mu'_{\rm old} - U \| - \| \mu_{\rm new} - \mu'_{\rm old}\|
= \left( 1 - \frac{1}{N} \right) \| \mu'_{\rm old} - U \| ,\]
by the equality in (\ref{eq0}).
Also, for any $k \in \{1, \ldots, N-1\}$,
\[ \| \mu_{\rm new} - X_k \| \leq \| \mu'_{\rm old} - X_k \| + \| \mu'_{\rm old} - \mu_{\rm new} \|
\leq D + \frac{1}{N} \| \mu'_{\rm old} - U \| ,\]
by (\ref{eq0}) again. Combining these estimates we obtain, for any $k \in \{1,\ldots,N-1\}$,
\[ \| \mu_{\rm new} - U \| - \| \mu_{\rm new} - X_k \| \geq \left( 1 - \frac{2}{N} \right) \| \mu'_{\rm old} - U \| - D
\geq \frac{1}{3} \| \mu'_{\rm old} - U \| - D,\]
for $N \geq 3$.
So in particular, $\| \mu_{\rm new} - U \| > \| \mu_{\rm new} - X_k \|$ for all $k \in \{1,\ldots, N-1\}$
provided $\| \mu'_{\rm old} - U \| > 3D$, i.e., $U \notin B ( \mu'_{\rm old} ; 3 D)$.
In this case, $U$ is the extreme point among
 $U, X_1, \ldots, X_{N-1}$,  i.e.,
\begin{equation}
\label{newextreme}
A_{t+1}^{\rm c} \subseteq \{ \XX_N^* (t+1) = U_{t+1} \} .\end{equation}
In particular, on $A_{t+1}^{\rm c}$, $F(t+1) = F(t)$, and
 $F(t+1) < F(t)$ only if $A_{t+1}$ occurs, giving the final inclusion in (\ref{events}). Since $\Pr [ A_{t+1} \mid \FF_t ]$ is
 bounded above by $C_d D^d$ for a constant $C_d<\infty$ depending only on $d$,
 (\ref{eq6}) follows from the first inequality in Lemma \ref{lem2}.
\end{proof}

\begin{lm}
\label{lem4}
Suppose that $N \geq 3$. Then, as $t\to\infty$, $F(t) \to 0$ a.s.\ and in $L^2$.
\end{lm}
\begin{proof}
Let $\eps >0$ and let $\sigma := \min \{ t \in \ZP : F (t) \leq \eps \}$, where $\ZP := \{0,1,2,\ldots\}$.
Then by (\ref{eq5}), there exists $\delta >0$ (depending on $\eps$ and $N$) such that,  a.s.,
$\Pr \left[ F (t+1)  - F ( t ) \leq - \delta \mid \FF_t \right] \geq \delta \1  {\{ t < \sigma \}}$.
Hence, since $F(t+1) - F(t) \leq 0$ a.s.\ by Corollary \ref{cor1},
\begin{equation}
\label{eq4} \Exp \left[ F (t+1)  - F ( t ) \mid \FF_t \right] \leq - \delta^2 \1 {\{ t < \sigma \}}. \end{equation}
By Corollary \ref{cor1}, $F(t)$ is   nonnegative  and nonincreasing, and hence $F(t)$ converges a.s.\ as $t \to \infty$ to some nonnegative limit $F(\infty)$;
the convergence also holds in $L^2$ since $F(t)$ is uniformly bounded. In particular, $\Exp [ F(t) ] \to \Exp [F (\infty) ]$.
So taking expectations in (\ref{eq4}) and letting $t \to \infty$ we obtain
\[  \limsup_{t \to \infty}  \delta^2 \Pr [ \sigma > t ] \leq 0 ,\]
which implies that $\Pr [ \sigma > t] \to 0$ as $t \to \infty$. Thus $\sigma < \infty$ a.s., which  together
with the monotonicity of $F(t)$ (Corollary \ref{cor1})
implies that $F(t) \leq \eps$ for all $t$ sufficiently large.
Since $\eps>0$ was arbitrary, the result follows.
\end{proof}

Recall the definition of $A_t$ and $A'_t$ from before Lemma \ref{lem3}.
Define $(\FF_t)$  stopping times   $\tau_0 := 0$ and, for $n \in \N$,   $\tau_n := \min \{ t > \tau_{n-1} :
A_{t} ~\textrm{occurs} \}$. Then $F (t) < F(t-1)$ can only occur if $t = \tau_n$ for some $n$.
Since $\Pr [ A_{t+1} \mid \FF_t ]$ is bounded below by a constant times $D(t)^d$, it is not hard to see that, provided $D(0) > 0$, $A_t$ occurs infinitely often, a.s., so that $\tau_n < \infty$ for all $n$.

\begin{lm}
\label{lem5}
Let $N \geq 3$. There exists $\alpha >0$ such that, a.s., $D(\tau_n) \leq   \re^{-\alpha n}$
for all $n$ sufficiently large.
\end{lm}
\begin{proof}
 We have from (\ref{eq2}) and the second inequality in Lemma \ref{lem2} that
 \[ F(\tau_n ) - F(\tau_n -1) \leq - \delta   F( \tau_n -1)   \1 ( A'_{\tau_n} )   ,\]
 for some $\delta >0$. Note also that, by definition
 of the stopping times $\tau_n$, $F(\tau_{n} - 1) = F(\tau_{n-1})$. Hence,
 \[ \Pr [ F(\tau_n ) - F(\tau_{n-1} ) \leq - \delta  F( \tau_{n-1}) \mid \FF_{\tau_{n-1}}  ]
 \geq \Pr [ A'_{\tau_n} \mid \FF_{\tau_{n-1}} ]  \geq \delta ,\]
 taking $\delta >0$ small enough, since, using the fact that $\1 (A_{\tau_n} ) =1$ a.s.,
\[ \Pr [ A'_{\tau_n} \mid \FF_{\tau_{n-1}} ] = \Exp \left[ \Pr [ A'_{\tau_n} \mid \FF_{\tau_n} ] \1 ( A_{\tau_n} ) \mid \FF_{\tau_{n-1}} \right]
= \Exp \left[ \Pr [ A'_{\tau_n} \mid A_{\tau_n} ] \mid \FF_{\tau_{n-1}} \right] ,\]
where
 by definition of $A_t$ and $A'_t$, $\Pr [ A'_{\tau_n} \mid A_{\tau_n} ]$ is uniformly positive.
Since $F(t+1) - F(t) \leq 0$ a.s.\ (by Corollary \ref{cor1}) it follows that
\[ \Exp \left[ F (\tau_{n}) - F(\tau_{n-1}) \mid \FF_{\tau_{n-1}}   \right] \leq -\delta^2 F(\tau_{n-1}) .\]
Taking expectations, we obtain $\Exp [ F (\tau_{n} )] \leq (1 -\delta^2 ) \Exp [ F ( \tau_{n-1} ) ]$,
which implies that $\Exp [ F (\tau_n ) ] = O ( \re^{-cn})$,  for some $c>0$ depending on $\delta$.
Then by Markov's inequality, $\Pr [ F( \tau_n ) \geq \re^{-cn/2} ] = O (\re^{-cn/2})$,
which implies that $F(\tau_n) = O ( \re^{-cn/2} )$, a.s., by the Borel--Cantelli lemma.
Then the first inequality in Lemma \ref{lem2} gives the result.
\end{proof}

\begin{rmk}
The proof of Lemma \ref{lem5} shows that $\Pr [ A'_{\tau_n} \mid \FF_{\tau_{n-1}} ]$ is uniformly positive,
so L\'evy's extension of the Borel--Cantelli lemma, with the fact that $\tau_n <\infty$ a.s.\ for all $n$,
shows that $A'_t$ occurs for infinitely many $t$, a.s. With the proof of Lemma \ref{lem3}, this shows that
$\XX_N^* (t) \neq U_t$ infinitely often, as claimed in Remark \ref{rem0}.
\end{rmk}

Now we are almost ready to complete the proof of Theorem \ref{thm1}. We state the main step in the remaining argument
as the first part of the the next lemma, while the second part of the lemma we will need
 in Section \ref{support} below.
For $\eps>0$, define the stopping time $\nu_\eps :=
 \min \{ t \in \N : F ( t ) <  \eps^2   \}$;
for any $\eps>0$, $\nu_\eps<\infty$      a.s.,  by Lemma \ref{lem4}.

\begin{lm}
\label{lem7}
Let $N \geq 3$. Then there exists $\xi_N \in [0,1]^d$ such that
$\mu_{N-1} ( \XX'_N (t) ) \to \xi_N$ a.s.\ and in $L^2$ as $t \to \infty$.
 Moreover, there exists an absolute constant $C$ such that
for any $\eps>0$, and any  $t_0 \in \N$,
on $\{ \nu_\eps \leq t_0 \}$, a.s.,
\[ \Exp \Big[ \max_{t \geq t_0} \left\| \mu_{N-1} ( \XX'_N (t) ) - \mu_{N-1} ( \XX'_N (t_0)) \right\|
\mid \FF_{t_0} \Big] \leq C \eps . \]
\end{lm}
\begin{proof}
Let $\mu' (t) := \mu_{N-1} ( \XX'_N(t) )$. Observe that for $N \geq 3$, $\XX'_N (t)$ and $\XX'_N(t-1)$ have at least one point
in common; choose one such point, and call it $Z(t)$. Then $\mu' (t) \in \hull  \XX_N'(t)   \subseteq \hull  \XX_N (t)$,
where $\hull \XX$ denotes the convex hull of the point set $\XX$. So $\| Z(t) - \mu'(t) \| \leq D(t)$.
Similarly $\| Z(t) - \mu'(t-1) \| \leq D(t-1)$.
By definition of $\tau_n$, $\mu'(t) = \mu'(t-1)$ and $D(t) = D(t-1)$ unless $t  = \tau_n$
for some $n$, in which case $\mu'(\tau_n -1 ) = \mu'(\tau_{n-1} )$ and $D(\tau_n - 1) = D(\tau_{n-1})$.
Hence,
\begin{align} \sum_{t \geq 1} \| \mu' (t) - \mu' (t-1) \| & =
\sum_{n \geq 1} \| \mu'(\tau_n) - \mu'(\tau_{n-1} ) \| \nonumber\\
& \leq \sum_{n \geq 1 } \left( \| \mu' (\tau_n ) - Z ( \tau_n ) \| +  \| \mu' (\tau_{n-1} ) - Z ( \tau_n ) \| \right) ,
\label{eq41}
\end{align}
by the triangle inequality. Then  the preceding remarks imply that
\[ \sum_{t \geq 1} \| \mu' (t) - \mu' (t-1) \|
   \leq \sum_{n \geq 1} ( D(\tau_n) + D(\tau_{n-1}) ) < \infty, ~{\rm a.s.}, \]
by Lemma \ref{lem5}. Hence there is some (random) $\xi_N \in [0,1]^d$ for which $\mu'(t) \to \xi_N$ a.s.\ as $t \to \infty$,
and $L^2$ convergence follows by the bounded convergence theorem.

For the final statement in the lemma we use a variation of the preceding argument.
Let $M := \max \{ n \in \ZP : \tau_n \leq t_0 \}$.
Then $F(t_0) = F(\tau_M)$ and $\mu' (\tau_M) = \mu' (t_0)$, so that 
on $\{ \nu_\eps \leq t_0 \}$, we have $\{ \nu_\eps \leq \tau_M \}$ as well. Hence
(by Corollary \ref{cor1}) 
$F ( \tau_M ) < \eps^2$.
A similar argument to that
in the proof of Lemma \ref{lem5} shows that, for $m \geq 0$,
\[ \Exp [ F(\tau_{M+m} ) \mid \FF_{t_0} ] \leq \re^{-cm} \Exp [ F(\tau_{M}) \mid \FF_{t_0} ] \leq  \eps^2  \re^{-cm} ,    \]
on $\{ \nu_\eps \leq t_0 \}$,
where $c>0$ depends on $N$ but not on $m$ or $\eps$. 
Thus by Lemma \ref{lem2}, on $\{ \nu_\eps \leq t_0 \}$,
$\Exp [ D ( \tau_{M+m} ) ^2 \mid \FF_{t_0} ] \leq 2  \eps^2  \re^{-cm}$. Also, similarly to (\ref{eq41}),
\begin{align*}
 \max_{t \geq \tau_M} \| \mu' (t) - \mu' (\tau_M) \|^2 & \leq \sum_{t \geq \tau_M} \| \mu'(t) - \mu'(t-1) \|^2 \\
 & \leq \sum_{m \geq 1} ( D (\tau_{M+m}) + D(\tau_{M+m-1} ) )^2 .\end{align*}
Taking expectations and using the Cauchy--Schwarz inequality, we obtain, on $\{ \nu_\eps \leq t_0 \}$,
\[ 
\Exp \Big[ \max_{t \geq t_0 } \| \mu' (t) - \mu' (t_0) \|^2 \mid \FF_{t_0} \Big] =
\Exp \Big[ \max_{t \geq \tau_M} \| \mu' (t) - \mu' (\tau_M) \|^2 \mid \FF_{t_0} \Big]
\leq 8   \eps^2 \re^c  \sum_{m \geq 1} \re^{-cm} ,\]
which is a constant times $\eps^2$.
The result follows from  Jensen's
inequality.
\end{proof}

\begin{proof}[Proof of Theorem \ref{thm1}.]
Again let $\mu' (t) := \mu_{N-1} ( \XX'_N(t) )$.
We have from Lemma \ref{lem7} that $\mu'(t) \to \xi_N$ a.s.
Now, for any $j \in \{1,\ldots, N-1\}$, by the triangle inequality,
\[ \| X_{(j)} (t) - \xi_N \| \leq \| X_{(j)} (t) - \mu' (t) \| + \| \mu'(t) - \xi_N \| \leq D (t) +   \| \mu'(t) - \xi_N \| ,\]
which tends to $0$ a.s.\ as $t \to \infty$, since $D(t) \to 0$ a.s.\ by Lemma \ref{lem5}.
 This establishes the first statement in (\ref{thm1a}). Moreover, by (\ref{newextreme}),
$\XX_N^*(t+1) \neq U_{t+1}$ only if $A_{t+1}$ occurs. On $A_{t+1}$, $\XX_N^*(t+1)$ is one of the points of $\XX_N'(t)$, and so in particular
$\| \XX_N^*(t+1) - \mu' (t) \| \leq D(t)$. In addition, on $A_{t+1}$, we have $\| U_{t+1} - \mu'(t) \| \leq 3 D(t)$. So by the triangle inequality,
\[ \left\| \XX_N^*(t+1) - U_{t+1} \right\| \leq 4 D(t) \1 ( A_{t+1} ) ,\]
  which tends to $0$ a.s., again by Lemma \ref{lem5}.
This gives the final part of (\ref{thm1a}).
\end{proof}

\section{The limit distribution in one dimension}
\label{sec:1d}

\subsection{Overview and simulations}

Throughout this section we restrict attention to $d=1$.
Of interest is the distribution of the limit $\xi_N$ in (\ref{thm1a}), and its behaviour as $N \to \infty$. Simulations 
suggest that $\xi_N$ is highly dependent on the initial configuration: Figure \ref{fig1}
shows histogram estimates for $\xi_N$ from repeated simulations with a deterministic initial condition.
In more detail, $10^8$ runs of each simulation were performed, each starting from the same initial condition; each run was terminated
when $D(t) < 0.0001$ for the first time, and the value of $\mu_{N-1} (\XX_N'(t))$  was output as an approximation to $\xi_N$ (cf 
Theorem \ref{thm1}). Note that, by (\ref{events}), in the simulations
one may take
the new points not $U[0,1]$ but uniform on a typically much smaller interval, which greatly increases the rate of updates to the core
configuration.

\begin{figure}[!h]
\begin{center}
\includegraphics[width=0.47\textwidth,clip=true,trim=0.65cm 1.6cm 0.8cm 0.8cm]{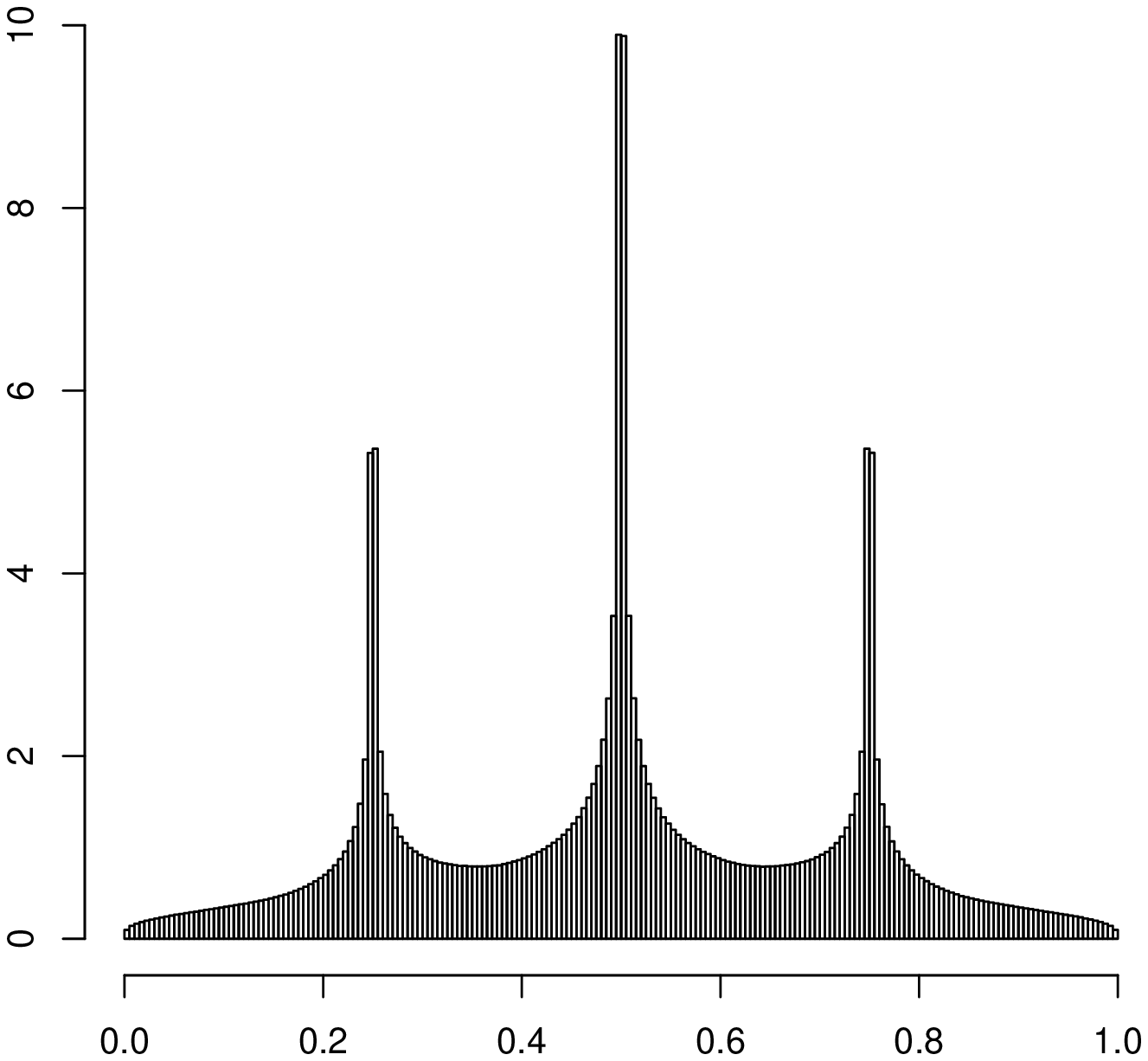}
\includegraphics[width=0.47\textwidth,clip=true,trim=0.65cm 1.6cm 0.8cm 0.8cm]{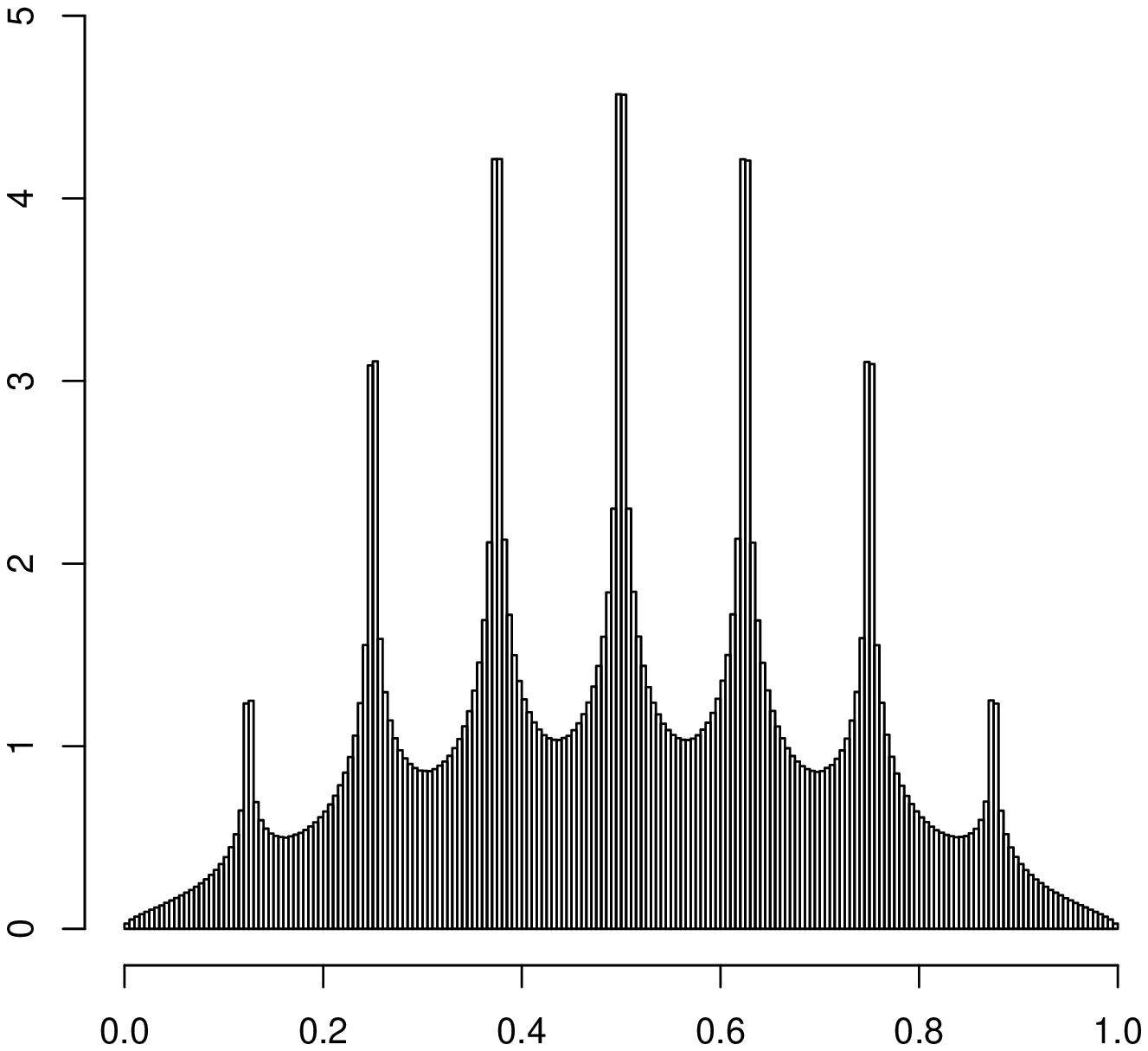}
\end{center}
\caption{Normalized histograms each based on $10^8$ simulations, with $N=3$ and initial points $\frac{1}{4},\frac{1}{2},\frac{3}{4}$ ({\em left})
and $N=7$ and initial points $\frac{k}{8}$, $k \in \{1,\ldots,7\}$ ({\em right}).}
\label{fig1}
\end{figure}

Figure \ref{fig2} shows sample results obtained
with an initial condition of $N$ i.i.d.\ $U[0,1]$ random points. Now the histograms appear much simpler, although, of course, they can be viewed as mixtures
of complicated multimodal histograms similar to
those in Figure \ref{fig1}.  
In the uniform case, it is natural to ask whether $\xi_N$ converge in distribution to some limit distribution as $N \to \infty$.

\begin{figure}[!h]
\begin{center}
\includegraphics[width=0.47\textwidth,clip=true,trim=0.65cm 1.6cm 0.8cm 0.8cm]{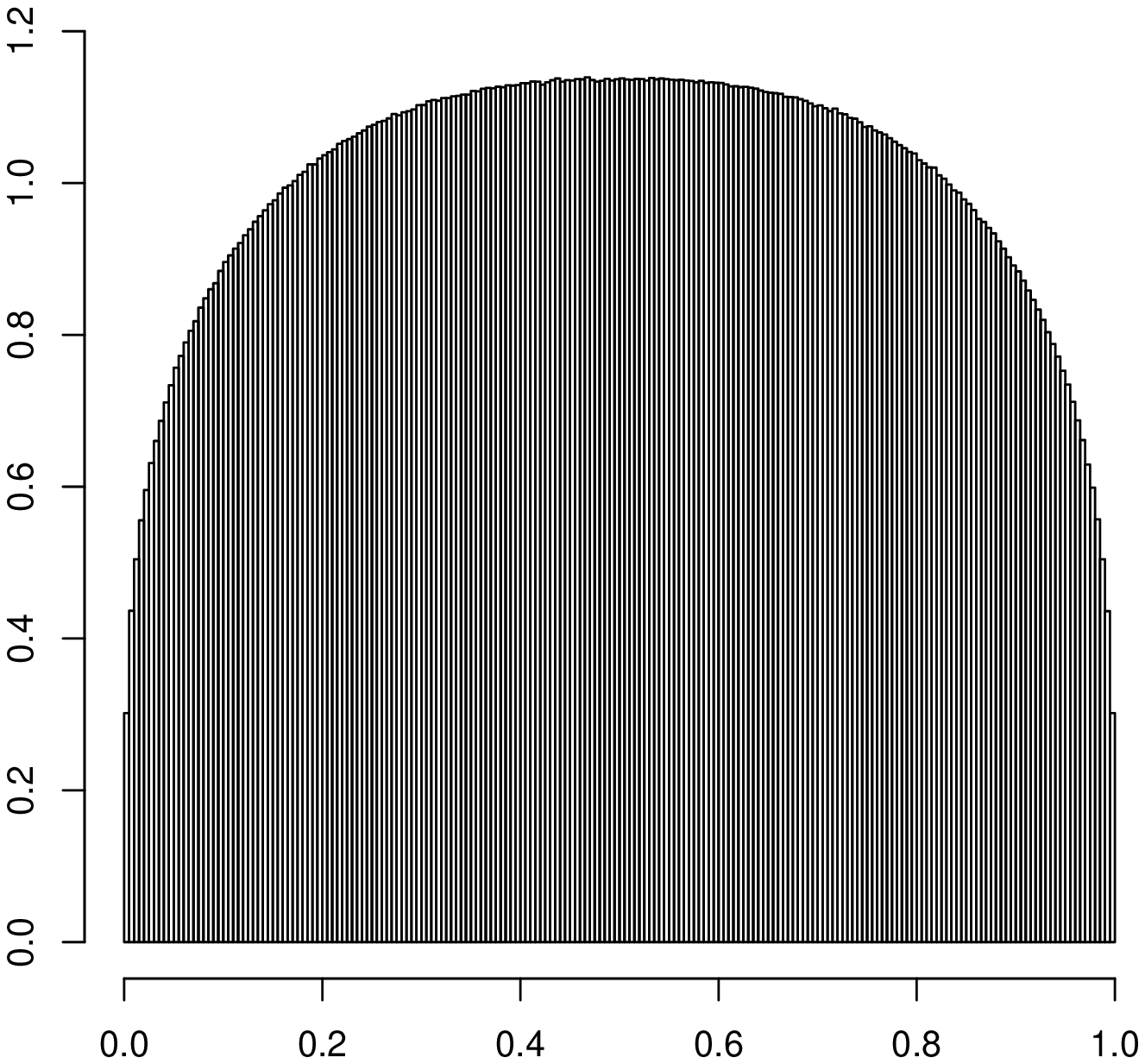}
\includegraphics[width=0.47\textwidth,clip=true,trim=0.65cm 1.6cm 0.8cm 0.8cm]{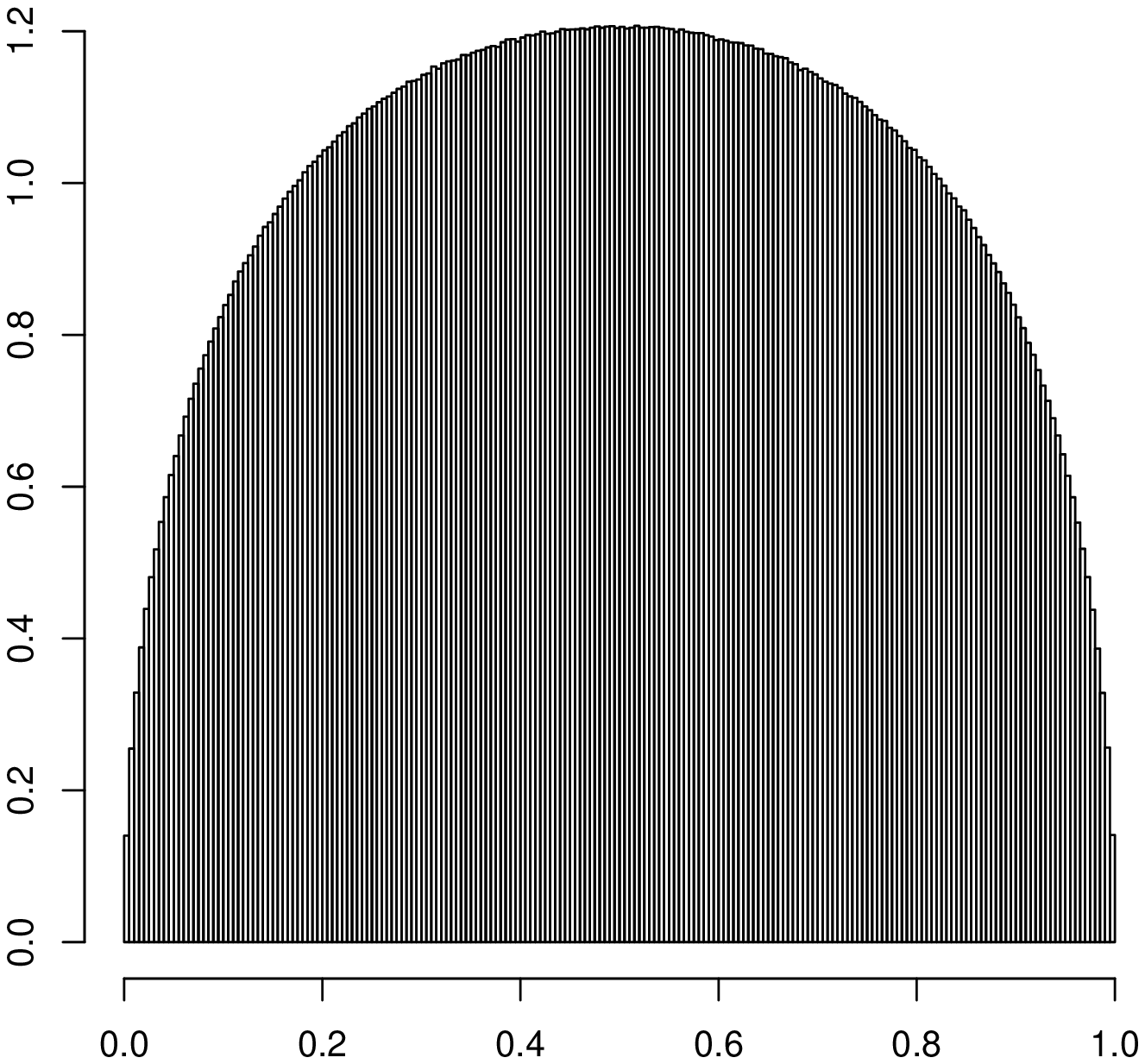}
\includegraphics[width=0.47\textwidth,clip=true,trim=0.65cm 1.6cm 0.8cm 0.8cm]{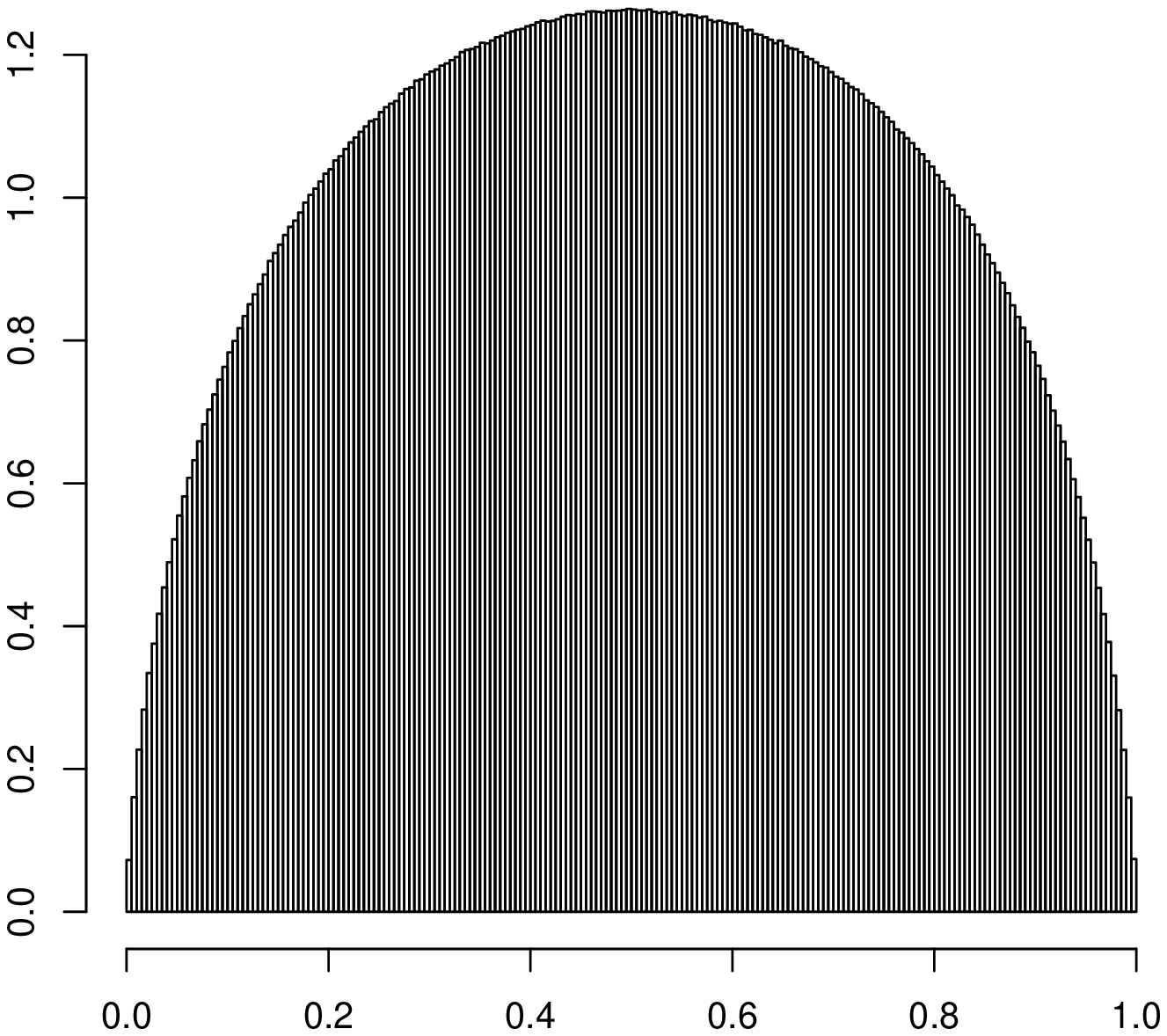}
\includegraphics[width=0.47\textwidth,clip=true,trim=0.65cm 1.6cm 0.8cm 0.8cm]{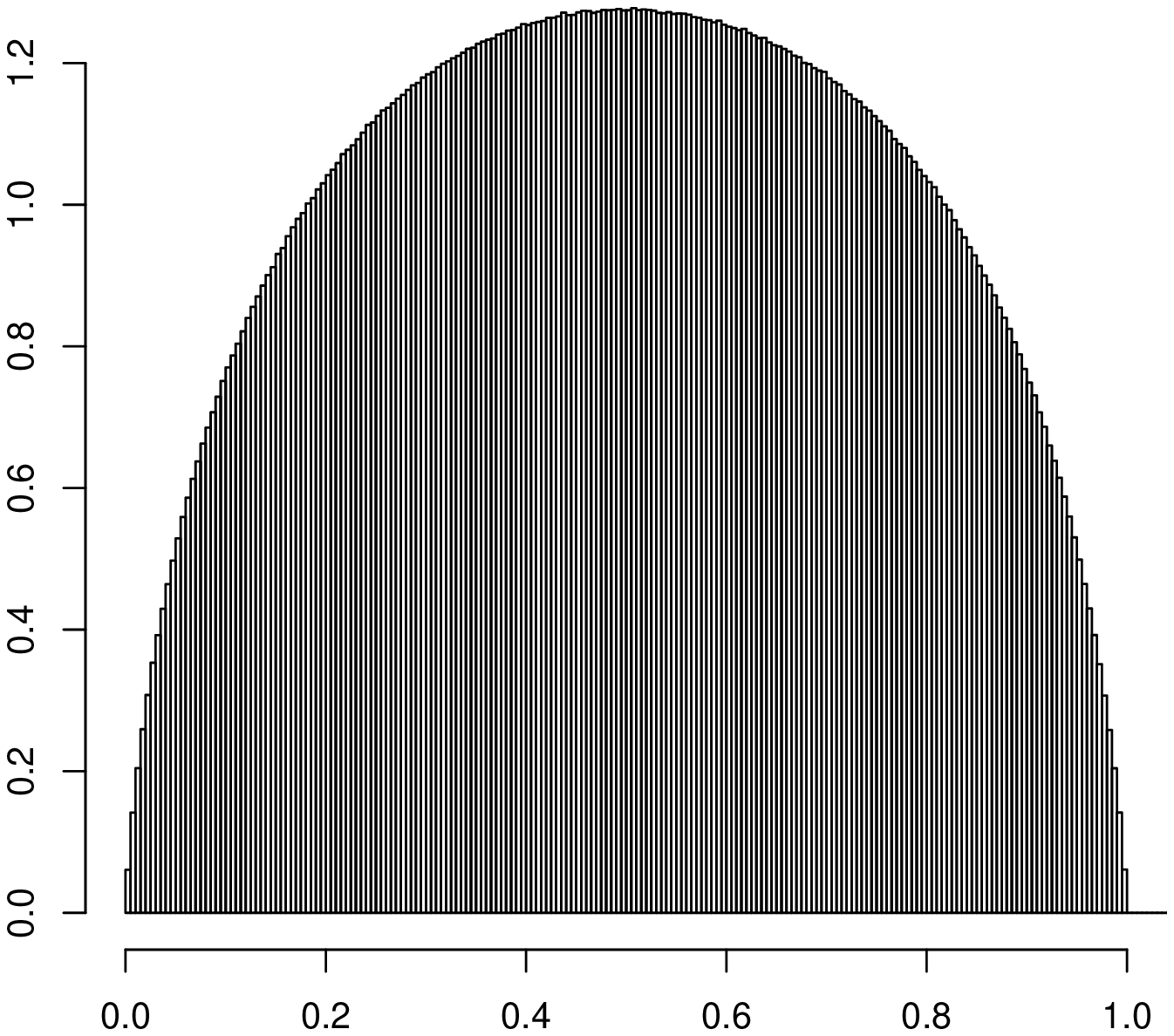}
\end{center}
\caption{Normalized histograms for $10^8$ simulations  with random i.i.d.\ uniform initial conditions, with ({\em top row}) $N=3, 10$ and ({\em bottom row}) $N=50, 100$.}
\label{fig2}
\end{figure}

The form of the histograms in Figure \ref{fig2}
might suggest a Beta distribution
(this is one sense in which the randomized beauty contest is ``reminiscent of a P\'olya urn'' \cite[p.\ 390]{gkw}).  
An ad-hoc Kolmogorov--Smirnov analysis (see Table \ref{tab1}) suggests that the distributions are indeed `close' to Beta distributions, but different
enough for the match to be unconvincing. Simulations for 
large $N$ are computationally intensive. We remark that it is not unusual for Beta or `approximate Beta'
distributions to appear as limits of schemes that proceed via iterated procedures on intervals: see for instance \cite{jk} and references therein.

\begin{table}[!h]
\center
\begin{tabular}{ccc}
$N$ & $\beta $ & $\kappa(\beta)$  \\
\hline
3   & 1.256      &  0.0010    \\
10  & 1.392      &  0.0016    \\
50  & 1.509      &  0.0018    \\
100 & 1.539      &  0.0019
\end{tabular}
\label{tab1}
\caption{$\kappa (\beta)$ is the Kolmogorov--Smirnov
distance between a Beta($\beta,\beta$) distribution and the empirical distribution from the samples of size $10^8$ plotted in Figure \ref{fig2}, minimized over $\beta$ in each case.}
\end{table}

In the rest of this section we study $\xi_N$ and its distribution. Our results on the limit distribution, in particular, leave several interesting open problems, including
a precise description of the phenomena displayed by the simulations reported above. In Section \ref{sec:rdp} we give an alternative (one might say
`phenomenological') characterization of the limit $\xi_N$, and contrast this with an appropriate rank-driven process  in the sense of \cite{gkw}.
In Section \ref{support} we show that the distribution of $\xi_N$ is fully supported on $(0,1)$ and assigns positive probability to any proper interval, using
a 
construction permitting transformations of configurations. Finally, Section \ref{sec:three} is devoted to the case $N=3$, for which
some explicit computations for the distribution of $\xi_N$ (in particular, its moments) are carried out.

\subsection{A characterization of the limit}
\label{sec:rdp}

Let
\[ \pi_N (t) := \frac{1}{t} \card \left\{ s \in \{1,2,\ldots, t\} : \XX_N^* (s) < \mu_{N} ( \XX_N (s) ) \right\} ,\]
the proportion of times up to time $t$ for which the extreme point was the {\em leftmost} point (as opposed to the rightmost).
The next result
shows that $\pi_N (t)$ converges to the (random) limit $\xi_N$ given by Theorem \ref{thm1}; we give the proof after some additional remarks.

\begin{prop}
\label{left}
Let $d=1$ and $N \geq 3$. Then $\lim_{t\to\infty}   \pi_N (t) = \xi_N$ a.s.
\end{prop}

It is instructive to contrast this behaviour with a suitable {\em rank-driven process} (cf \cite{gkw}).
Namely, fix a parameter $\pi \in (0,1)$. Take $N$ points in $[0,1]$, and at each step in discrete time replace either the leftmost  point (with probability $\pi$)
or else the rightmost point (probability $1-\pi$), independently at each step; inserted points are independent $U[0,1]$ variables. For this process,
results of \cite{gkw} show that the marginal distribution
 of a typical point converges (as $t \to \infty$ and then $N \to \infty$) to a unit point mass
at $\pi$ (cf Remark 3.2 in \cite{gkw}).

This leads us to one sense in which the randomized beauty contest is, to a limited extent, ``reminiscent of a P\'olya urn'' \cite[p.\ 390]{gkw}.
 Recall that a P\'olya urn consists of an increasing number of balls, each of which is either red or blue; at each step in discrete time, a ball is drawn uniformly
at random from the urn and put back into the urn together with an extra ball of the same colour. The stochastic process of interest is the proportion of red balls, say; it converges
to a random limit $\pi'$, which has a Beta distribution.
The beauty contest can be viewed as occupying a similar relation to the rank-driven process  described above as the  P\'olya urn process
does to the simpler model in which, at each step, independently,
either
  a red ball is added to the urn (with probability $\pi'$) or else a blue ball
 is added (probability $1-\pi'$).

\begin{proof}[Proof of Proposition \ref{left}.]
Given $\tau_0, \tau_1, \tau_2, \ldots$, $\xi_N = \lim_{n \to \infty} \mu_{N-1} ( \XX'_N (\tau_n) )$ is independent
of $U_t$, $t \notin \{ \tau_0, \tau_1, \ldots \}$, since, by (\ref{newextreme}),
 any such $U_t$ is replaced at time $t+1$.
Let $\eps>0$. By Theorem \ref{thm1}, there exists a random $T < \infty$ a.s.\ for which
$\max_{1 \leq i \leq N-1} | \xi_N - X_{(i)}   (t)   | \leq \eps  $ for all $t \geq T$.

Since $\mu_N (\XX_N (t+1)) = \frac{N-1}{N} \mu_{N-1} ( \XX'_N (t) ) + \frac{1}{N} U_{t+1}$, we have that for $t \geq T$, using the
triangle inequality, for any $i \in \{1,\ldots, N-1\}$,
\begin{align*}  | \mu_N ( \XX_N (t+1) ) - X_{(i)} (t) |  & \leq \eps + | \mu_N ( \XX_N (t+1) ) - \xi_N | \\
& \leq \eps + \frac{N-1}{N} | \mu_{N-1} ( \XX'_N (t) ) - \xi_N | + \frac{1}{N} | U_{t+1} - \xi_N | .\end{align*}
Hence, for $t \geq T$,  
\begin{equation}
\label{88a}
\max_{1 \leq i \leq N-1}
\left| X_{(i)} (t) -  \mu_{N} ( \XX_N (t+1) ) \right|
\leq \frac{1}{N} | \xi_N - U_{t+1} | + 2 \eps.
\end{equation}
On the other hand, for $t \geq T$, $\mu_{N} ( \XX_N (t+1) ) \geq \frac{N-1}{N} ( \xi_N - \eps) + \frac{1}{N} U_{t+1}$,
so that for $i \in \{ 1, \ldots,  N-1\}$,
\begin{equation}
\label{88b} \mu_{N} ( \XX_N (t+1) ) -  U_{t+1}   \geq \frac{N-1}{N} ( \xi_N - U_{t+1}   -\eps )  .\end{equation}
Suppose that $U_{t+1} < \xi_N - K \eps$ for some $K \in (1,\infty)$.
Then, from (\ref{88a}) and (\ref{88b}),
\begin{align*}
& ~~{} \left| \mu_N ( \XX_N (t+1) ) - U_{t+1} \right| - \max_{1 \leq i \leq N-1} \left| X_{(i)} (t) - \mu_N ( \XX_N (t+1) ) \right|
\\
& \geq \frac{N-2}{N} ( \xi_N - U_{t+1} ) - \frac{3N-1}{N} \eps \\
& > \frac{\eps}{N} \left( (N-2) K - 3N +1 \right ) .\end{align*}
This last expression is positive
 provided $K \geq \frac{3N-1}{N-2}$, which is the case for all $N \geq 3$ with the choice $K=8$, say.
Hence, with this choice of $K$, $\{ U_{t+1} < \xi_N - 8 \eps \}$
implies that $U_{t+1}$ is farther from $\mu_{N+1} ( \XX_N (t+1) )$ than is
any of the points left over from $\XX'_N(t)$.
Write $L_t := \{  U_{t}  < \xi_N - 8 \eps \}$.
Then we have shown that, for $t \geq T$,
 the event $L_t$ implies that $U_{t} = \XX^*_N (t)$, and, moreover, $U_{t} < \mu_{N} ( \XX_N (t) )$.
Hence, for $t \geq T$,
\[  \pi_N (t) \geq \frac{1}{t} \sum_{s=T}^t \1 ( L_s )   \geq    \frac{1}{t} \sum_{\stackrel{s=T}{s \notin \{ \tau_0, \tau_1, \ldots\}}}^t \1 ( L_s ).
\]
Given $\tau_0, \tau_1, \ldots$,
$U_s$, $s \notin \{ \tau_0, \tau_1, \ldots\}$ are independent of $T$ and $\xi_N$.
For such an $s$, $U_s$ is uniform on $I_s := [0,1] \setminus B ( \mu_{N-1} ( \XX'_N (s) ) ; 3 D(s) )$,
and, for $s \geq T$, $D(s) \leq 2 \eps$ so that
 $I_s \supseteq [0, \max \{ \xi_N - 8 \eps , 0\} ] \cup [ \min\{ \xi_N + 8   \eps, 1\} , 1]$.
Hence, given  $s \notin \{ \tau_0, \tau_1, \ldots\}$ and $s \geq T$,
\[ \Pr [ L_s ] = \Pr [ U_{s}  < \xi_N - 8 \eps \mid U_s \in I_s ]
\geq  
\xi_N - 8 \eps .
   \]
Hence, considering separately the cases $\xi_N > 9 \eps$ and $\xi_N \leq 9 \eps$, the strong law
of large numbers implies that
\[ \frac{1}{t} \sum_{\stackrel{s=T}{s \notin \{ \tau_0, \tau_1, \ldots\}}}^t \1 ( L_s ) \geq  \xi_N - 9 \eps ,\]
for all $t$ sufficiently large;
here we have used the fact that $t -T \to \infty$ a.s.\ as $t \to \infty$ and   $\# \{ n \in \ZP : \tau_n \leq t \} = o(t)$ a.s., which follows
 from (\ref{eq6}) and Lemma \ref{lem4}.
 Since $\eps>0$ was arbitrary,
it follows that $\liminf_{t \to \infty} t^{-1} \pi_N (t) \geq \xi_N$ a.s.
The symmetrical argument considering events of the form $R_t := \{U_t > \xi_N + 8 \eps\}$
shows that $\liminf_{t \to \infty} ( 1 - t^{-1}  \pi_N (t) ) \geq 1 - \xi_N$ a.s., so $\limsup_{t \to \infty} t^{-1} \pi_N (t) \leq \xi_N$ a.s.
Combining the two bounds gives the result.
\end{proof}

\subsection{The limit has full support}
\label{support}

 In this section, we prove  that $\xi_N$ is fully supported on $(0,1)$ in the sense that $\essinf \xi_N = 0$, $\esssup \xi_N =1$,
  and $\xi_N$ assigns positive probability to any non-null interval.  
  Let
  \begin{equation}
  \label{spacing}
   m_N (0) := \min \{ \| X_i(0) - X_j(0) \| : i,j \in \{0,1,\ldots,N+1\}, \, i \neq j \}, \end{equation}
  where we use the conventions $X_0(0) := 0$ and $X_{N+1} (0) := 1$.
 For $\rho >0$ let $S_\rho$ denote the $\FF_0$-event  
$S_\rho := \{ m_N (0) \geq  \rho \}$ that no point of $\XX_N(0)$ is closer than distance $\rho$
to any other point of $\XX_N(0)$ or to either of the ends of the unit interval.

\begin{prop}
\label{limit}
 Let $d =1$ and $N \geq 3$. Let $\rho \in (0,1)$. For any
 non-null interval subset $I$ of $[0,1]$, there exists $\delta >0$
 (depending on $N$, $I$, and $\rho$) for which
 \begin{equation}
 \label{lowerbound}
  \Pr [ \xi_N \in I \mid \XX_N (0) ] \geq  \delta \1 ( S_\rho) , \as \end{equation}
 In particular, in the case where $\XX_N(0)$ consists of $N$ independent
  $U[0,1]$ points,
 $\Pr [ \xi_N \in I ] >0$ for any non-null interval $I \subseteq [0,1]$.
\end{prop}

  We suspect, but have not been able to prove,
 that $\xi_N$ has a density $f_N$ with respect to Lebesgue
measure, i.e., $\xi_N$ is absolutely continuous in the sense that for every $\eps>0$ there exists $\delta>0$ such that
$\Pr [ \xi_N \in A ] < \eps$ for every $A$ with Lebesgue measure less than $\delta$. Were  this
so,
then Proposition \ref{limit} would show that we may take $f_N(x) >0$ for all $x \in (0,1)$. Note that $\Pr [ \xi_N \in A \mid \XX_N (0) ]$
may be $0$  if $\XX_N (0)$ contains non-distinct points: e.g.\ if $N \geq 3$ and $\XX_N(0) = (x,x,\ldots,x,y)$, then
$\XX'_N(t) = (x,x,\ldots,x)$ for all $t$.

For $a \in [0,1]$, $\eps >0$, and $t \in \N$, define the event
\[ E_{a,\eps} ( t ) := \bigcap_{i=1}^N \left\{ | X_i (t) - a | < \eps \right\} .\]
The main new ingredient needed to obtain Proposition \ref{limit} is the following result.

\begin{lm}
\label{move}
Let $N \geq 3$.
For any $\rho \in (0,1)$  and   $\eps >0$  there exist $t_0 \in \N$ and $\delta_0 >0$
(depending on $N$,  $\rho$, and $\eps$) for which, for all $a \in [0,1]$,
\[ \Pr [ E_{a,\eps} (t_0) \mid \XX_N(0) ] > \delta_0 \1 (S_\rho ) , \as \]
\end{lm}
\begin{proof}
Fix $a \in [0,1]$.
Let $\rho \in (0,1)$ and $\eps>0$. It suffices to suppose that $\eps \in (0,\rho)$, since $E_{a,\eps} (t) \subseteq E_{a,\eps'} (t)$ for $\eps' \geq \eps$.
Suppose that $S_\rho$ occurs, so that $m_N (0) \geq \rho$ with $m_N(0)$ defined at (\ref{spacing}).
For ease of notation we list the points of $\XX_N(0)$ in increasing order as $0<X_1< X_2 < \dots < X_N<1$.
Let $M = \lfloor N/2\rfloor$.
 
Let $\nu=\eps/N^2$.
The following argument shows how one can arrive at a configuration at a finite (deterministic)
time $t_0$ where all of $X_1(t_0),\ldots,X_N(t_0)$ lie inside $(a-\eps,a+\eps)$
with a positive (though possibly very small) probability.

Let us call the points which are present at time $0$ \emph{old points}; the points which will gradually replace this set will be called \emph{new points}.
We will first describe an event by which all the old points are removed and replaced by new points arranged approximately equidistantly
 in the interval $[X_M, X_{M+1}]$, and then we will describe an event by which such a configuration can migrate to the target interval.

{\sc Step} 1.
Starting from time $0$, iterate the following procedure until a new point becomes an extreme point.
The construction is such that at each step, the extreme point is one of the old points, either at the
 extreme left or right of the configuration. At each step, the extreme old point is removed and replaced
 by a new $U[0,1]$ point to form the configuration at the next time unit. We describe an event of positive probability by requiring the successive
 new arrivals to fall in particular intervals, as follows.
The first old point removed from the \emph{right} is replaced by a new point in $( X_M+\nu,X_M+\nu+\delta )$, where $\delta \in (0,\nu)$
will be specified later.
Subsequently, the $i$th point ($i \geq 2$) removed from the right is replaced by a new point
 in $( X_M+i\nu,X_M+i\nu+\delta )$.
We call this subset of new  points the \emph{accumulation on the left}. On the other hand,
the $i$th extreme point removed from the {\em left} ($i \in \N$) is replaced by a new point
in $( X_{M+1}-i \nu,X_{M+1}- i \nu+\delta )$.
 This second subset of new points will be called  the \emph{accumulation on the right}.

During the first $M$ steps of this procedure, the new points are necessarily internal points
of the configuration and so are never removed. 
Therefore, there will be a time $t_1\in [M,N]$ at which, for the first time,
 one of the new points becomes either the leftmost or rightmost
point of $\XX_N (t_1)$; suppose that it is the rightmost, since  the argument in the other case is analogous. 
If at time $t_1$ the accumulation on the right is non-empty, we continue to perform the procedure described in
 {\sc Step} 1, but now allowing ourselves to remove new
points from the accumulation on the right. So we continue putting extra points on the accumulation on the left whenever the rightmost point is removed, and similarly putting extra points to the accumulation on the right whenever the leftmost point is removed, as described for
 {\sc Step} 1.
Eventually we will have either (a) a configuration where all the new points of the left or the right accumulation are completely removed, and there are still some of the old points left, or (b) a configuration where all old points are removed.
The next step we describe separately for these two possibilities.

{\sc Step} 2(a). Without loss of generality, suppose that the accumulation on the right is empty, so the configuration
consists of  $k$ points of the left accumulation and $N-k$  old points remaining to the left of $X_M$ (including $X_M$ itself).
Note that {\sc Step} 1 produces at least $M$ new points, so $M \leq k \leq N-1$, since by assumption we have at least one old point remaining.
Let us now denote the  points of the configuration  $x_1 < x_2 < \cdots < x_N$ so that $x_{N-k} = X_M$, and
by the construction in {\sc Step} 1, $x_{N-k+i} \in ( X_M+i\nu,X_M+i\nu+\delta )$ for $i=1,2,\dots,k$.
Provided that $k \leq N-2$, so that there are at least 2 old points, we will show that
$x_1$ is necessarily the extreme point of the configuration.
Indeed, writing $\mu=\mu_N(x_1,\ldots,x_N)$, using the fact that $x_{N-k+i} \geq x_{N-k} + i \nu$ for
$1 \leq i \leq k$ and $x_{N} \leq x_{N-k} + k \nu + \delta$, we have
\begin{align*}
\mu -\frac{x_1+x_N}2 & 
\geq \frac{x_1+\cdots+x_{N-k}+kx_{N-k}+ \frac{1}{2} \nu k(k+1)  }{N}
 - \frac{x_1+ x_{N-k}+\nu k+ \delta}{2} \\
& = 
\frac{1}{2N} \left( 2 x_1 + \cdots + 2 x_{N-k} + (2k-N) x_{N-k} - N x_1 +   
 \nu k(k+1 -N) - \delta N \right) .\end{align*}
The old points all have separation at least $\rho$, so for $1 \leq i \leq N-k$, $x_i \geq x_1 + (i-1) \rho$,
and hence
\[ 2 x_1+\cdots+2 x_{N-k}+(2 k - N) x_{N-k} \geq N x_1 +   \rho (N-k-1) (N-k) +   \rho (2k- N) (N-k-1) .\]
It follows, after simplification, that
\begin{align*}
\mu -\frac{x_1+x_N}2
&\ge \frac{1}{2N} \left(    k (N-k-1) (\rho - \nu) - \delta N    \right) \\
& \geq \frac{1}{2N} \left( (N-2) (\rho - \nu ) - \delta N \right) ,
\end{align*}
provided $1 \leq k \leq N-2$. By choice of $\nu$, we have $\nu \leq \rho /9$ and it follows that the last
displayed expression is positive provided $\delta$ is small enough compared to $\rho$ ($\delta < \rho/4$, say).
 Hence $|x_1-\mu| > |x_N-\mu|$.
 Thus next we remove $x_1$. We replace it similarly to the procedure in {\sc Step} 1, but now building up the accumulation on the {\em left}.
We can thus iterate this step, removing old points from the left and building up the accumulation on the left, while
keeping the accumulation on the right empty, until we get just one old point remaining (i.e.\ until $k=N-1$); this last old
point will be $X_M$. At this stage, after a finite number of steps,
 we end up with a configuration where  the set of points $x_1<x_2<\dots<x_N$ satisfies
$x_i\in [ X_M +(i-1)\nu ,X_M + (i-1) \nu + \delta ]$, $i=1,2,\dots,N$.

{\sc Step} 2(b).
Suppose that the configuration is such that all old points have been removed
 but both left and right accumulations are non-empty.
 Repeating the procedure of {\sc Step} 1, replacing rightmost points by building the left accumulation and leftmost points by building the right accumulation,
   we will also, in a finite number of steps, obtain  a set points $x_i$ such that
   $x_i\in [ b+(i-1) \nu,b+(i-1) \nu+\delta ]$, $i=1,2,\dots,N$, for some $b\in [0,1]$.

{\sc Step} 3.
Now we will show how one can get to the situation where all  points
 lie inside the interval $(a-\eps,a+\eps)$
starting from any configuration in which
\begin{align}\label{eqx*}
x_i \in [ b+i\nu,b+i\nu+\delta ],\ i=1,2,\dots,N-1 ,
\end{align}
where $b \in [0,1]$ and $x_1 < \cdots < x_{N-1}$ are the core points of the configuration (i.e., with
  the extreme point removed).
  We have shown in {\sc Step} 1 and {\sc Step} 2 how we can achieve such a configuration in a finite time with a positive probability.
  Suppose that $a > b$; the argument for the other case is entirely analogous. We describe an event of positive probability
  by which the entire configuration can be moved to the right.

Having just removed the extreme point, we stipulate that the new point
 $y_1$ belong  to $( b+N\nu-6\delta,b+N\nu-5\delta)$, so $y_1 > x_{N-1}$ is the new
 rightmost point provided $\delta < \nu/7$.
 Then to ensure that $x_1$, and not $y_1$, is the most extreme point we need
$$
\frac{x_1+y_1}2  -\left[b+\nu\frac{N+1}2\right]< \frac{x_1+\dots+x_{N-1}+y_1}N -\left[b+\nu\frac{N+1}2\right] .
$$
The left-hand side of the last inequality is less than $-2\delta$
while the right-hand side is more than $ -\frac{6\delta}N$, so the inequality
is indeed satisfied provided $N\ge 3$.

\begin{figure}[!h]
\center
\setlength{\unitlength}{1mm}
\begin{picture}(150,20)(-2,0)
 \shortstack[c]{
\thicklines
\put(0, 10){\line(1, 0){150}}
\put(-5,4){$b+\nu$}
\put(0, 12){\line(0, -1){4}}
\put(0, 12){\vector(1, 0){7}}
\put(7, 12){\vector(-1, 0){7}}
\put(3, 15){$\delta$}
\put(3,10){\circle*{2}}
\put(35,4){$b+2\nu$}
\put(40, 12){\line(0, -1){4}}
\put(40, 12){\vector(1, 0){7}}
\put(47, 12){\vector(-1, 0){7}}
\put(43, 15){$\delta$}
\put(43,10){\circle*{2}}
\put(65,4){$\cdots$}
\put(95,4){$b+(N-1)\nu$}
\put(100, 12){\line(0, -1){4}}
\put(100, 12){\vector(1, 0){7}}
\put(107, 12){\vector(-1, 0){7}}
\put(103, 15){$\delta$}
\put(103,10){\circle*{2}}
\put(135,4){$b+N\nu$}
\put(140, 12){\line(0, -1){4}}
\put(140, 12){\vector(-1, 0){20}}
\put(120, 12){\vector(1, 0){20}}
\put(140, 12){\vector(-1, 0){13}}
\put(120, 12){\vector(1, 0){7}}
\put(122, 15){$\delta$}
\put(131, 15){$5\delta$}
\put(123,10){\circle*{2}}
}
\end{picture}
\caption{Schematic of a configuration at the start of {\sc Step} 3. The  disks represent the points $x_1, x_2, \ldots, x_{N-1}$ and,
on the extreme right, the new point $y_1$.}
\label{fig3}
\end{figure}
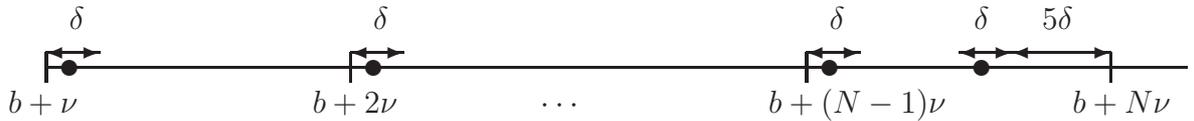

Hence at the next step $x_1$ is removed.
Our new collection of core points is $x_2 < \cdots < x_{N-1} < y_1$.
We stipulate that the next new point $y_2$ arrive in $( b+(N+1)\nu-18\delta,b+(N+1)\nu -17\delta)$. So again, for $\delta$ small enough
($\delta <  \nu /13$ suffices), $y_2 > y_1$ and the
newly added point ($y_1$) becomes the rightmost point in the configuration.
Again, to ensure that the leftmost point ($x_2$)
is now the extreme one, we require
$$
\frac{x_2+y_2}2  -\left[b+\nu\frac{N+3}2\right]< \frac{x_2+\cdots+x_{N-1}+y_1+y_2}N -\left[b+\nu\frac{N+3}2\right].
$$
The left-hand side of the last inequality is less than $-8\delta$, while the right-hand side is more
 than $-24\delta/N$, and so the displayed inequality is true provided $N\ge 3$.

We will repeat this process until we remove the rightmost core point present at the start of {\sc Step} 3, namely $x_{N-1}$,
located in $[ b+(N-1)\nu,b+(N-1)\nu+\delta ]$.
We will demonstrate how we can do this, in succession removing points from the left of the configuration and at each step replacing
them by points on the right with careful choice of locations for the new points.
We consecutively put new points $y_k$ at locations in intervals
$$
\Delta_k:=( b+(N-1+k)\nu-2\cdot 3^k \delta,b+(N-1+k)\nu-2\cdot 3^k \delta+\delta ) ,
$$
for $k=1,2,\dots, N-1$.
We have just shown that for $k=1,2$ this procedure will maintain the leftmost point ($x_k$) 
as the extreme one.
Let us show that this is true for all $1 \leq k \leq N-1$, by an inductive argument. Indeed, suppose that the original
points $x_1,x_2,\dots,x_{k-1}$ have been removed,  the successive
 new points $y_j$
 are located in $\Delta_j$, $j=1,2,\dots,k-1$, and that the replacement for the most recently removed point
 $x_{k-1}$ is the new point $y_k$.  Place the new point $y_{k}$ in $\Delta_{k}$. Provided
 $\delta < \frac{\nu}{4 \cdot 3^{k-1} +1}$, $y_k > y_{k-1}$ and $y_k$
 is the rightmost point of the new configuration,
 while the leftmost point is $x_{k}\in [ b+\nu k,b+\nu k+\delta ]$. Since $N\ge 3$ we have
\begin{align*}
\frac{x_{k}+y_{k}}{2}&\le b+\left[\frac{N-1}{2}+k\right]\nu- [3^{k}-1]\delta
\le b+\left[\frac{N-1}{2}+k\right]\nu- \frac{2(3+3^2+\dots+3^{k})\delta}{N}\\
&\le \frac{x_{k}+\dots+x_{N-1}+y_1+\dots+y_{k}}{N} ,
\end{align*}
thus ensuring that the leftmost point $x_k$, and not $y_{k}$, is the farthest from the centre of mass.

Thus, provided $\delta <  3^{-N} \nu$, say, we proceed to remove all the points $x_k$ and
  end up with a new collection of points $x_1',\dots,x_{N-1}'$ satisfying the property
$$
x_i'\in [ b'+i\nu,b'+i\nu+\delta'],\ i=1,2,\dots,N-1 ,
$$
where $b'=b+(N-1)\nu- \delta'$ and $\delta':=3^N\delta$ ($>2\cdot 3^{N-1}\delta)$.
Thus the situation is similar to the one in~(\ref{eqx*}) but with $b$ replaced by $b' > b + \nu$,
so the whole ``grid'' is shifted to the right.
Hence, provided $\delta$ is small enough, and $\delta'$ and its subsequent analogues remain such that
$\delta' < 3^{-N} \nu$, we can repeat the above procedure and move points to the right again, etc., a finite number of times
(depending on $|b-a|/\nu$)
 until the moment when all the new points are indeed in $(a-\eps,a+\eps)$, and the probability of making all those steps is strictly positive.
In particular, we can  check that taking $ \delta < 3^{-2N/\nu} \nu$ will suffice.

All in all, we have performed a finite number of steps,
which can be
bounded above in terms of $N$, $\rho$, and $\eps$ but independently of $a$,
and each of which
required a $U[0,1]$ variable to be placed in a small interval (of width less than $3^{-N} \nu$) and so has
positive probability, which can be bounded below in terms of $N$, $\rho$, and $\eps$.
So overall the desired transformation of the configuration has positive probability
depending on $N$, $\rho$, and $\eps$, but not on $a$.
\end{proof}

\begin{proof}[Proof of Proposition \ref{limit}.]
Write $\mu' (t) := \mu_{N-1} ( \XX'_N(t) )$. Let $I \subseteq [0,1]$
be a non-null interval. We can (and do) choose $a \in (0,1)$ and $\eps' >0$
such that $I' := [a-\eps' , a+\eps' ] \subseteq I$.
Also take $I'' := [a -\eps , a+ \eps  ] \subset I'$ for $\eps = (4 B C)^{-1} N^{-1/2} \eps'$,
where $C$ is the constant in   Lemma \ref{lem7} and $B \geq 1$ is an absolute constant chosen so that $\eps < \eps'/4$ for all $N \geq 3$.
Fix $\rho \in (0,1)$.
It follows from Lemma \ref{move} that, for some $\delta_0 >0$ and $t_0 \in \N$,
depending on $\eps$,
\begin{equation}
\nonumber
\Pr [ \{ \mu' (t_0) \in I'' \} \cap \{ D(t_0) \leq 2 \eps  \} \mid \XX_N(0) ] \geq \delta_0 \1 ( S_\rho) . \end{equation}
  By Lemma \ref{lem2}, we have that $D(t_0) \leq 2 \eps $
implies that $F(t_0) \leq 2 N \eps^2 <  ( \eps' / (2BC) )^2$, so that
$t_0 \geq \nu_{\eps'/(2BC)}$, where $\nu_\cdot$ is as defined just before Lemma \ref{lem7}.
Applying Lemma \ref{lem7} with this choice of $t_0$ and with the $\eps$ there equal to $\eps'/(2 B C)$, we obtain, by Markov's inequality,
\[ \Pr \left[ \max_{t \geq t_0 }  | \mu' (t) - \mu' ( t_0 )  | \leq 3\eps'/4 \mid \FF_{t_0} \right] \geq 1/3 , \as \]
 It follows that, given $\XX_N(0)$, the event
\[ \{ \mu' (t_0) \in I'' \} \cap \{ D(t_0) \leq 2 \eps \}  \cap \{
   | \xi_N - \mu' ( t_0 )  | \leq  3 \eps' /4 \} \]
  has probability at least $(\delta_0 /3) \1 ( S_\rho)$,
  and on this event we have $| \xi_N - a| \leq \eps + (3\eps'/4) < \eps'$, so $\xi_N \in I$.
Hence  (\ref{lowerbound}) follows.

For the final statement in the proposition, suppose that $\XX_N(0)$ consists of independent $U[0,1]$ points.
In this case $m_N(0)$ defined at (\ref{spacing}) is the minimal spacing in the induced partition
of $[0,1]$ into $N+1$ segments, which has the same distribution as $\frac{1}{N+1}$ times a single spacing,
and in particular has density $f(x) =  N (N+1)  (1- (N+1) x)^{N-1}$ for $x \in [0, \frac{1}{N+1}]$ (cf Section \ref{appendix}).
Hence for any $\rho \in [0,\frac{1}{N+1}]$, we have $\Pr [ S_\rho ] =   (1-(N+1) \rho)^N$,
which is positive for   $\rho = \frac{1}{2N}$, say. Thus taking expectations in (\ref{lowerbound}) yields the final statement
in the proposition.
\end{proof}

\subsection{Explicit calculations for $N=3$}
\label{sec:three}

For this section we take $N=3$, the smallest nontrivial example.
In this case we can perform some explicit calculations to obtain information about
 the   distribution of $\xi_3$.
In fact, we work with a slightly modified version of the model, avoiding certain `boundary effects', 
to ease  computation.
Specifically, we do not use $U[0,1]$ replacements but,
given $\XX_3(t)$, we take $U_{t+1}$ to be uniform on the interval $U[ \min \XX'_3 (t) - D(t) , \max \XX'_3 (t) + D(t) ]$.
If this interval is contained in $[0,1]$ for all $t$, this modification would have no effect on the value of $\xi_3$ realized
(only speeding up the   convergence), but the fact that now $U_{t+1}$ might be outside $[0,1]$ {\em does} change the model.
 
For this  modified model, the argument for Theorem \ref{thm1}
 follows through with minor changes, although we essentially reprove
the conclusion of Theorem \ref{thm1} in this case when we prove the following result, which gives an explicit description of the limit distribution.
Here and subsequently `$\stackrel{d}{=}$' denotes equality in distribution.

\begin{prop}
\label{prop3}
Let $d = 1$ and $N = 3$ and work with the modified version of the process just described. Let $\XX_3(0)$ consist of $3$ distinct points in $[0,1]$.
Write $\mu := \mu_2 ( \XX_3'(0) )$ and $D := D_2(\XX_3'(0))$.
There exists a random $\xi_3 := \xi_3 (\XX_3(0)) \in \R$ such that
$\XX'_3(t) \toas ( \xi_3, \xi_3)$  as $t \to \infty$. The distribution of $\xi_3$ can be characterized via $\xi_3 \eqd \mu + D L$, where $L$
is independent of $(\mu, D)$, $L \eqd -L$, and the distribution of $L$ is determined by the distributional solution to the fixed-point equation
\begin{equation}
\label{Lfixed} L  \eqd \begin{cases}
 - \frac{1+U}{2} + U L & \textrm{with~probability~} \frac{1}{3} \\
  - \frac{2-U}{4} + \frac{U}{2} L & \textrm{with~probability~} \frac{1}{6} \\
 \phantom{-}    \frac{2-U}{4} + \frac{U}{2} L & \textrm{with~probability~} \frac{1}{6} \\
 \phantom{-}     \frac{1+U}{2} + U L & \textrm{with~probability~} \frac{1}{3} ,\end{cases} \end{equation}
    where $\Exp [ |L|^k ] < \infty$ for all $k$, and $U \sim U[0,1]$. Writing $\theta_k := \Exp [ L^k]$, we have $\theta_k =0$ for odd $k$, and
$\theta_2 = \frac{7}{12}$, $\theta_4 = \frac{375}{368}$, and $\theta_6 = \frac{76693}{22080}$.
In particular,
\begin{equation}
\label{xi2}
\Exp[ \xi_3 ] = \Exp [ \mu] , ~ \Exp [ \xi_3^2 ] = \Exp [ \mu^{ 2} ] + \frac{7}{12} \Exp [ D^{ 2} ],
~\textrm{and}~   \Exp [ \xi_3^2 ] = \Exp [ \mu^{ 3} ] + \frac{7}{4} \Exp [ \mu D^{ 2} ].\end{equation}
In the case where  $\XX_3(0)$  contains $3$ independent $U[0,1]$ points,   $\Exp [ \xi_3^k ] = \frac{1}{2}, \frac{1}{3}, \frac{1}{4}$ for $k = 1,2,3$ respectively.
If $\XX_3(0) = (   \frac{1}{4}, \frac{1}{2}, \frac{3}{4} )$, then
$\Exp [ \xi_3^k ]= \frac{1}{2},
\frac{29}{96}, \frac{13}{64}, \frac{873}{5888}$ for $k =1,2,3,4$.
\end{prop}

We give the proof of Proposition \ref{prop3} at the end of this section. First we
state one consequence of the fixed-point representation (\ref{Lfixed}).

\begin{prop}
\label{Lcontinuous}
 $L$ given by (\ref{Lfixed}) has an absolutely continuous distribution.
\end{prop}
\begin{proof}
It follows from (\ref{Lfixed}) that
\begin{align*}
\Pr\left[ L=\tfrac 12\right ] &=
\tfrac 13 \, \Pr\left[U\left(L-\tfrac 12 \right)=1\right]
+\tfrac 16 \,\Pr\left[U\left(L+\tfrac 12 \right)=2\right]\\ 
& ~~{ }+\tfrac 16 \,\Pr\left[U\left(L-\tfrac 12 \right)=0\right]
+\tfrac 13 \,\Pr\left[U\left(L+\tfrac 12 \right)=0\right] .\end{align*}
The first two terms on the right-hand side of the last display are
zero, by an application of the first part of Lemma~\ref{lemconts1}
with $X=U$, $Y=L \pm 1/2$, and $a=1,2$. Also, since $U >0$ a.s.,
$\Pr [ U (L \mp \frac{1}{2}) = 0 ] = \Pr [ L = \pm \frac{1}{2}]$, and, by symmetry, 
$\Pr [ L=1/2 ]=\Pr [ L=-1/2]$.
Thus we obtain
 \begin{align*}  \Pr\left[ L=\tfrac 12\right ]
 = \tfrac 16 \,\Pr\left[ L=\tfrac 12 \right]
+\tfrac 13 \,\Pr\left[L=-\tfrac 12 \right]
=\tfrac 12 \,\Pr\left[L=\tfrac 12 \right] .
\end{align*}
Hence  $\Pr[L=1/2]=\Pr[L=-1/2]=0$.

Each term on the right-hand side of (\ref{Lfixed}) is
of the form $\pm \frac{1}{2} + V ( L \pm \frac{1}{2} )$
where $V$ is an absolutely continuous random variable, independent of $L$
(namely $U$ or $U/2$).
The final statement in Lemma~\ref{lemconts1} with the
fact that $\Pr [ L = \pm 1/2]=0$ shows that 
each such term is  absolutely continuous.
   Finally,    Lemma~\ref{lemconts2}
 completes the proof.
\end{proof}

In principle, the characterization (\ref{Lfixed}) can be used to recursively determine all the moments $\Exp [ L^k] =\theta_k$,
and the moments of $\xi_3$ may then be obtained by expanding $\Exp [ \xi_3 ^k ] = \Exp [ (\mu + D L)^k ]$.
However, the calculations soon become cumbersome, particularly as $\mu$ and $D$ are, typically, not independent:
we   give some distributional properties of $(\mu, D)$ in the case of a uniform random initial condition in Section \ref{appendix}.

Before giving the proof of Proposition \ref{prop3}, we comment on some simulations. Figure \ref{fig4}
shows histogram estimates for the distribution of $\xi_3$ for two initial distributions
(one deterministic and the other uniform random), and Table \ref{tab2} reports corresponding moment estimates, which may be compared to the
theoretical values given in Proposition \ref{prop3}. In the uniform case, we only computed the first $3$ moments analytically,
namely,
$\frac{1}{2}$, $\frac{1}{3}$, $\frac{1}{4}$ as quoted in Proposition \ref{prop3}; it is a curiosity that these coincide with the first 3
moments of the $U[0,1]$ distribution.

\begin{figure}[!h]
\begin{center}
\includegraphics[width=0.48\textwidth,clip=true,trim=0.65cm 1.6cm 0.8cm 0.8cm]{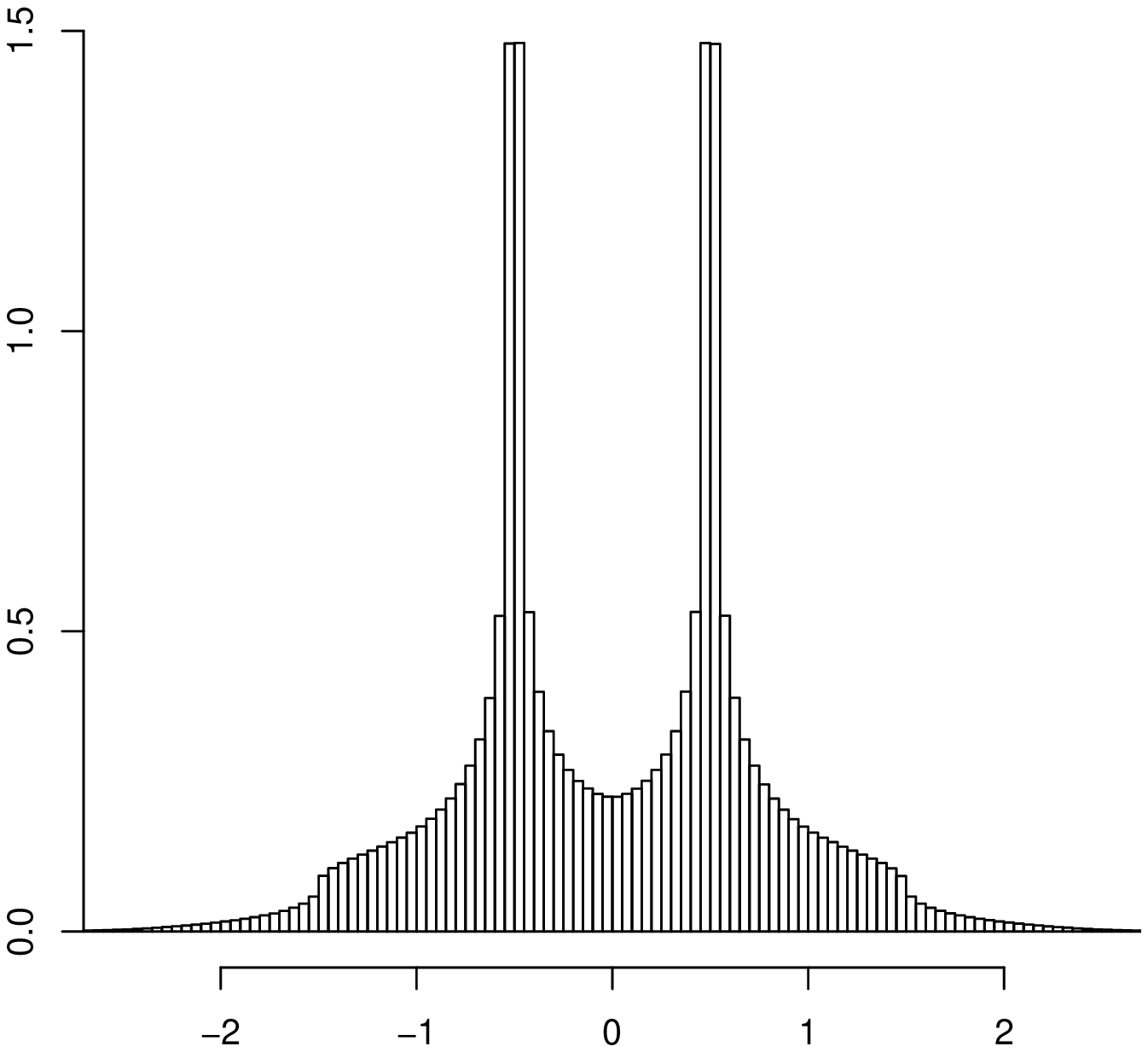}
\includegraphics[width=0.48\textwidth,clip=true,trim=0.65cm 1.6cm 0.8cm 0.8cm]{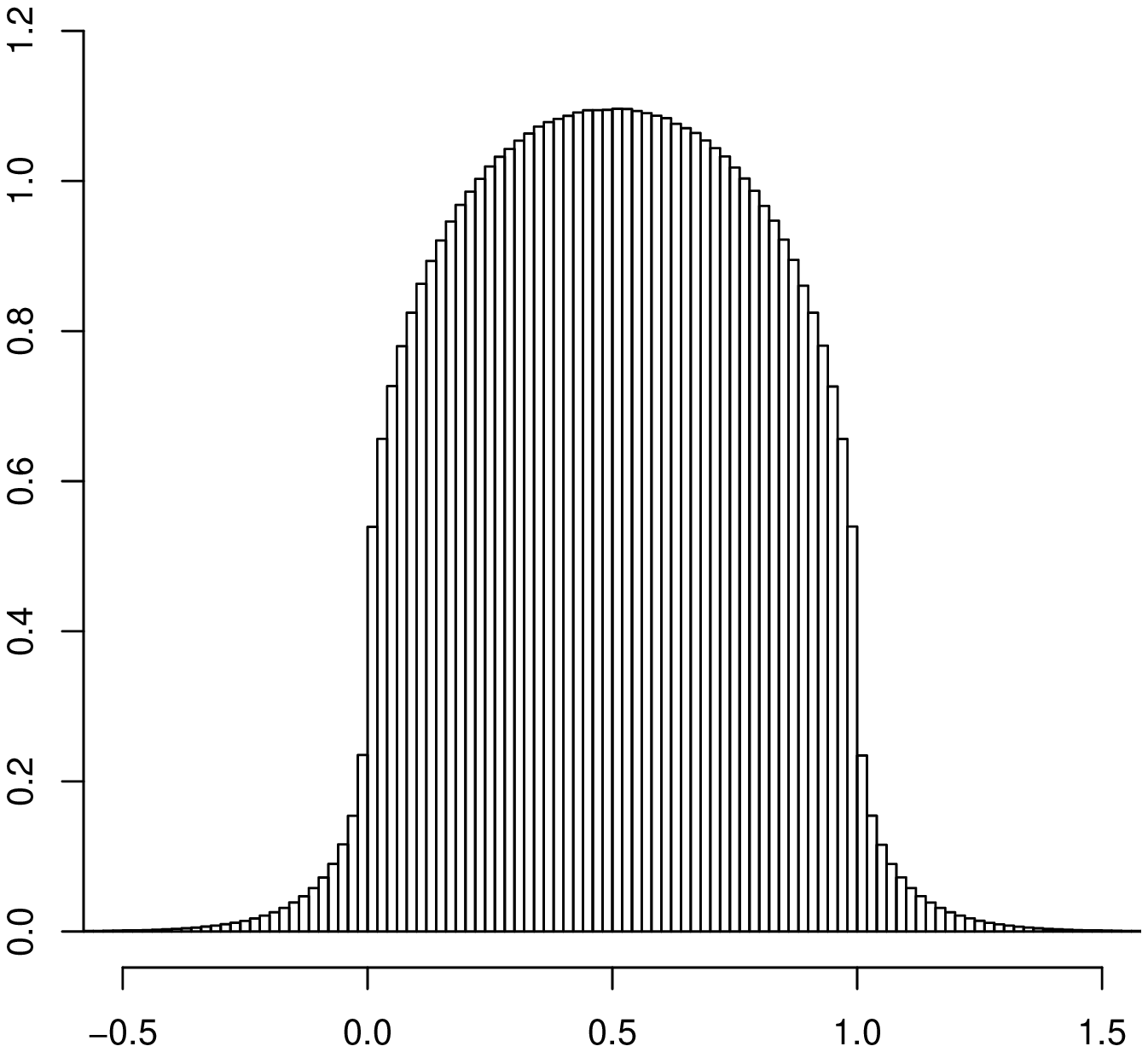}
\end{center}
\caption{Normalized histograms   based on $10^8$ simulations  of the modified $N=3$ model with fixed $\{-1/2,1/2,100\}$ initial condition ({\em left}) and
  i.i.d.\ $U[0,1]$ initial condition ({\em right}).}
\label{fig4}
\end{figure}

\begin{table}[!h]
\begin{center}
\begin{tabular}{c|cccccc}
$k$ & 1 & 2 & 3 & 4 & 5 & 6\\
\hline
\textrm{$\pm \frac{1}{2}$ core} & 0.0001 & 0.5833 & 0.0000 & 1.0192 & $-0.0005$ & 3.4765 \\
U[0,1]   & 0.5000 & 0.3333 & 0.2500 & 0.2029 & 0.1739 & 0.1561
\end{tabular}
\caption{
Empirical $k$th moment values  (to $4$ decimal places) computed from the simulations in Figure \ref{fig4}. }
\label{tab2}
\end{center}
\end{table}

\begin{proof}[Proof of Proposition \ref{prop3}.]
Let $\mu'(t) := \frac{1}{2} ( X_{(1)} (t) + X_{(2)} (t) )$ and $D(t) := | X_{(1)} (t) - X_{(2)} (t) |$
denote the mean and diameter of the core configuration, repeating our notation from above.

Consider  separately the events that $U_{t+1}$ falls in each of the intervals
$[ \min \XX'_3 (t) - D(t) ,  \min \XX'_3 (t))$, $[\min \XX'_3 (t) , \min \XX'_3 (t) + \frac{1}{2} D(t) )$,
$[ \min \XX'_3 (t) + \frac{1}{2} D(t) , \max \XX'_3 (t) )$,
$[ \max \XX'_3(t) ,  \max \XX'_3(t) + D(t)]$, which have probabilities $\frac{1}{3}, \frac16, \frac 16, \frac 13$ respectively.
Given $(\mu'(t), D(t))$, we see, for $V_{t+1}$ a $U[0,1]$ variable, independent of $(\mu'(t), D(t))$,
\begin{align}
\label{recur}
 ( \mu'(t+1) , D(t+1) ) & = \begin{cases}
(\mu'(t) - \frac{1+V_{t+1}}{2} D(t) , V_{t+1} D(t) ) & \textrm{ with probability } \frac{1}{3} \\
(\mu'(t) - \frac{2-V_{t+1}}{4} D(t) , \frac{1}{2} V_{t+1} D(t) )  & \textrm{ with probability } \frac{1}{6} \\
(\mu'(t) + \frac{2-V_{t+1}}{4} D(t) , \frac{1}{2} V_{t+1} D(t) )  & \textrm{ with probability } \frac{1}{6} \\
(\mu'(t) + \frac{1+V_{t+1}}{2} D(t) , V_{t+1} D(t) ) & \textrm{ with probability } \frac{1}{3}
\end{cases} .\end{align}
Writing
 $m_{k} (t) = \Exp [ D(t)^k \mid \XX_3 (0) ]$ we obtain from the second coordinates in (\ref{recur})
\[ m_k (t+1) = \frac{2}{3} \Exp [ V_{t+1}^k] m_k (t) + \frac{1}{3} 2^{-k} \Exp [ V_{t+1}^k] m_k (t) ,\]
 which implies that
\begin{equation}
\label{mk}
 m_k (t) = \left( \frac{1}{3(k+1)} ( 2 + 2^{-k} ) \right)^t D(0)^k .\end{equation}
For example, $m_1 (t) = (5/12)^t D(0)$ and $m_2(t) = (1/4)^t D(0)^2$.

Next we show that $\mu'(t)$ converges. From (\ref{recur}), we have that $| \mu' (t+1) - \mu' (t) | \leq D(t)$, a.s.,
so to show that $\mu'(t)$ converges, it suffices to show that $\sum_{t=0}^\infty D(t) < \infty$ a.s.
But this can be seen from essentially the same argument as
Lemma \ref{lem5}, or directly from the fact that the sum has nonnegative terms and  $\Exp \sum_{t=0}^\infty D(t) = \sum_{t=0}^\infty \Exp [ m_1(t) ]$,
which is finite. Hence $\mu'(t)$ converges a.s.\ to some limit, $\xi_3$ say.
Extending this argument a little, we have from (\ref{recur}) that $| \mu'(t+1) | \leq | \mu'(t) | + D(t)$, a.s., and $D(t+1) \leq V_{t+1} D(t)$, a.s.
Hence for $U_1, U_2, \ldots$ i.i.d.\ $U[0,1]$
random variables, we have $D(t) \leq V_1 \cdots V_t D(0)$ and
\[ | \mu'(t) | = \sum_{s=0}^{t-1} ( | \mu' (s+1)| -| \mu'(s) | ) \leq  \left( 1  +  \sum_{s=1}^\infty \prod_{r=1}^s V_r \right) D(0) =: ( 1 + Z ) D(0).\]
Here $Z$ has the so-called {\em Dickman distribution} (see e.g.\ \cite[\S3]{pw1}), which has finite moments of all orders.
 Hence $\Exp [ | \mu'(t)|^p \mid \XX_3 (0) ]$
 is   bounded independently of $t$, so, for any $p \geq 1$, $(\mu'(t))^p$ is uniformly integrable, and hence
$\lim_{t \to \infty} \Exp [ ( \mu' (t) )^k \mid \XX_3 (0) ] = \Exp [ \xi_3^k \mid \XX_3 (0) ]$ for any $k \in \N$.

We now want to compute the moments of $\xi_3$; by the previous argument, we can first work with the moments of $\mu'(t)$.
Note that, from (\ref{recur}),
\begin{align*} &~~ \Exp [ (\mu'(t+1)  -  \mu'(t))^{ k} \mid \XX_3 (t) ] \\
& = \frac{1+(-1)^k}{3} D(t)^{ k} \Exp \left[ \left( \frac{1+V_{t+1}}{2} \right)^{ k} \right]
+ \frac{1+(-1)^k}{6} D(t)^{ k}  \Exp \left[ \left( \frac{1+V_{t+1}}{4} \right)^{ k} \right] \\
& = \frac{1+(-1)^k}{6} ( 2^{1- k} + 2^{-2k} ) \frac{2^{ k+1} -1}{ k+1} D(t)^{ k} ,\end{align*}
using the fact that $\Exp [ (1+V_{t+1})^{ k} ] = \frac{2^{ k+1} -1}{ k+1}$.
In particular, $\Exp [ ( \mu'(t+1) - \mu'(t) )^k \mid \XX_3 (t) ] =0$ for odd $k$,
so $\Exp [ \mu'(t) \mid \XX_3 (0) ] = \mu'(0)$, and hence $\Exp [ \xi_3 ] = \lim_{t \to \infty} \Exp [ \mu'(t) ] = \Exp [ \mu'(0) ]$,
giving
the first statement in (\ref{xi2}).
In addition,
\begin{align*} & ~~ \Exp [ (\mu'(t+1))^2  -  (\mu'(t))^{ 2} \mid \XX_3 (t) ] \\
& = 2 \mu'(t) \Exp[  \mu'(t+1)  -  \mu'(t)  \mid \XX_3 (t) ] + \Exp [ (\mu'(t+1)  -  \mu'(t))^{ 2} \mid \XX_3 (t) ] \\
& = \frac{7}{16} D(t)^{ 2} .\end{align*}
Hence
\begin{align*} \Exp [ ( \mu'(t) )^{ 2} - (\mu'(0))^{ 2} \mid \XX_3 (0) ]
 & =   \sum_{s=0}^{t-1} \Exp [ ( \mu'(s+1))^2 - (\mu'(s))^2  \mid \XX_3 (0) ]  \\
&  = \frac{7}{16}  \sum_{s=0}^{t-1} m_{ 2} (s)  \\
&  \to \frac{7}{16} \sum_{s=0}^\infty 4^{-s} D(0)^2  ,\end{align*}
as $t \to \infty$, and the limit evaluates to $\frac{7}{12} D(0)^2$, so that
$\Exp [ \xi_3^2 ] = \lim_{t \to \infty} \Exp [ (\mu'(t))^2 ] = \Exp [ (\mu'(0))^2 ] + \frac{7}{12} \Exp [ D(0) ^2 ]$, giving the
second statement in (\ref{xi2}).

Write $L ( \mu'(0), D(0) ) = \xi_3 (\XX_3 (0))$ emphasizing the dependence on the
initial configuration through $\mu'(0)$ and $D(0)$. Then by translation and scaling properties
\begin{equation}
\label{scaling} L ( \mu'(0) , D(0) ) \eqd \mu'(0) + D(0) L ( 0, 1) .\end{equation}
So we work with $L:= L(0,1)$ (which has the initial core points at $\pm \frac{1}{2}$).

We will derive a fixed-point equation for $L$. The argument is closely related to that for (\ref{recur}).
Conditioning on the first replacement and using the transformation relation (\ref{scaling}), we obtain (\ref{Lfixed}).
From (\ref{Lfixed}) we see that $|L|$ is stochastically dominated by
$1 + U  | L |$; iterating this, similarly
 to the argument involving the Dickman distribution above,
 we obtain that $|L|$ is stochastically dominated by $1+Z$, where $Z$ has the Dickman distribution, which
 is determined by its moments. Hence (\ref{Lfixed}) determines a unique distribution for $L$ with $\Exp [ |L|^k ] < \infty$ for all $k$.

Writing (\ref{Lfixed}) in functional form $L \eqd \Psi ( L)$, we see that by symmetry of the form of $\Psi$, also $\Psi ( L) \eqd - \Psi ( - L)$.
Hence $-L \eqd - \Psi (L) \eqd \Psi (-L)$, so $-L$ satisfies the same distributional fixed-point equation as does $L$. Hence $L \eqd -L$.

Writing $\theta_k := \Exp [ L^k]$, which we know is finite, we get
\begin{align*} \theta_k & = \frac{1}{3} \sum_{j=0}^k {\binom k j} ( 1 + (-1) )^j \theta_{k-j} \Exp \left[ \left( \frac{1+U}{2} \right)^j U^{k-j} \right] \\
& ~~ {} + \frac{1}{6} \sum_{j=0}^k \binom k  j ( 1 + (-1) )^j \theta_{k-j} \Exp \left[ \left( \frac{2-U}{4} \right)^j \left( \frac{U}{2}\right)^{k-j} \right] .\end{align*}
 Here
 \begin{align*} \Exp \left[ \left( \frac{1+U}{2} \right)^j U^{k-j} \right] & = 2^{-j} \sum_{\ell = 0}^j {\binom j  \ell} \frac{1}{k-\ell +1 } =: a(k,j) ; \\
\Exp \left[ \left( \frac{2-U}{4} \right)^j \left( \frac{U}{2}\right)^{k-j} \right] & = 2^{-k} \sum_{\ell = 0}^j \binom j \ell \frac{(-1/2)^{j-\ell}}{k-\ell +1 } =: b(k,j).
\end{align*}
So we get
\begin{equation}
\label{thetak}
 \theta_k = \frac{1}{3} \sum_{j \textrm{ even}, \, j \leq k } \binom k  j  \theta_{k-j} ( 2 a (k,j) + b (k,j) ) .\end{equation}
 In particular, as can be seen either directly by symmetry or by an inductive argument using (\ref{thetak}), $\theta_k = 0$ for odd $k$. For even $k$, one can use
(\ref{thetak}) recursively to find $\theta_k$, obtaining for example the values quoted in the proposition.

Note that, by (\ref{scaling}), $\Exp [ \xi_3^3 ] = \Exp [ ( \mu'(0) + L D(0) )^3 ]$, which, on expansion, gives the final statement in (\ref{xi2}).
The first 3 moments in the case of the uniform initial condition  follow  from (\ref{xi2}) and Lemma \ref{initlem}.
For the initial condition with points $\frac{1}{4}, \frac{1}{2}, \frac{3}{4}$, we have $D(0) = \frac{1}{4}$
and $\mu'(0) = \chi \frac38 + (1-\chi) \frac58 = \frac58- \frac \chi 4$, where $\chi$ is the tie-breaker
random variable taking values $0$ or $1$ each with probability $\frac{1}{2}$. It follows that
$\Exp [ \mu'(0)^k] = \frac{1}{2} 8^{-k} (3^k + 5^k)$. Then, using (\ref{scaling}),
\[ \Exp [ \xi_3^k ] = \Exp [ ( \mu'(0) + (L/4))^k ] = \frac{1}{2} 8^{-k} \sum_{j=0}^k {\binom k  j} 2^j \theta_j (3^{k-j}+5^{k-j})  .\]
We can now compute the four moments given in the proposition.
\end{proof}

\section{Appendix 1: Uniform spacings}
\label{appendix}

In this appendix we collect some results about uniform spacings which allow us to obtain
 distributional results
about our uniform initial configurations. The basic results that we build on here can be found in Section 4.2 of \cite{pwong}; see the
references therein for a fuller treatment of the theory of spacings.

Let $U_1, U_2, \ldots, U_n$ be independent $U[0,1]$ points. Denote the corresponding
increasing order statistics $U_{[1]} \leq \cdots \leq U_{[n]}$, and define the induced
{\em spacings} by $S_{n,i} := U_{[i]} - U_{[i-1]}$, $i=1,\ldots,n+1$,
with the conventions $U_{[0]} := 0$ and $U_{[n+1]} := 1$.
We collect some basic facts about the $S_{n,i}$.
The spacings are exchangeable, and any $n$-vector, such as $(S_{n,1}, \ldots, S_{n,n})$,
has the uniform density on the simplex $\Delta_n := \{ (x_1, \ldots, x_n) \in [0,1]^n : \sum_{i=1}^n x_i \leq 1 \}$.

We   need some joint properties of up to 3 spacings. Any 3 spacings have density $f(x_1,x_2,x_3) = n (n-1)(n-2) (1-x_1-x_2-x_3)^{n-3}$
on $\Delta_3$. We will  make use of the facts
\begin{align}
\label{min1}
\min \{ S_{n,1} ,S_{n,2} \} & \eqd \tfrac{1}{2} S_{n,1} , ~~~ (n \geq 1) , \\
( S_{n,1} , \min \{ S_{n,2} ,S_{n,3} \} ) & \eqd ( S_{n,1} , \tfrac{1}{2} S_{n,2} ) , ~~~ (n \geq 2) ;
\label{min2}
\end{align}
see for example Lemma 4.1 of \cite{pwong}. Finally,
 for any $n \geq 1$ and $\alpha \geq 0, \beta \geq 0$,
\begin{equation}
\label{multispace}
\Exp [ S_{n,1}^\alpha S_{n,2}^\beta   ] = \frac{\Gamma (n+1) \Gamma (\alpha +1 ) \Gamma (\beta + 1) }{\Gamma (n + 1 +\alpha + \beta  )} .\end{equation}
In particular $\Exp [ S_{n,1}^k ] = \frac{n!k!}{(n+k)!}$ for   $k \in \N$.

Our main application in the present paper of the results on spacings collected above is to obtain the following result, which we   use
in Section \ref{sec:three}.

\begin{lm}
\label{initlem}
Let $d=1$ and $N=3$. Suppose that $\XX_3(0)$ consists of $3$ independent $U[0,1]$ points.
Then
\[ ( \mu_2 (\XX'_3 (0) ) , D_2 (\XX'_3(0)) ) \eqd ( (S_1 + \tfrac{1}{4} S_2 ) \zeta + (1 - S_1 - \tfrac{1}{4} S_2 ) (1 -\zeta) , \tfrac{1}{2} S_2 ) ,\]
where $\zeta$ is a Bernoulli random variable with $\Pr [ \zeta = 0] = \Pr[ \zeta = 1 ] = 1/2$.
For any $k \in \ZP$,
\begin{align}
\label{momd}
  \Exp [ ( D_2 (\XX'_3(0)) )^k ] & = 2^{-k} \frac{6}{(k+1)(k+2)(k+3)} , \\
  \label{mommu}
  \Exp [ ( \mu_2 (\XX'_3 (0) ) )^k ] & = \frac{4 (3k - 5 + (3^{k+3} - 1) 4^{-(k+1)} )}{(k+1)(k+2)(k+3)} .
  \end{align}
  So, for example, the first 3 moments of $D_2 (\XX'_3(0))$ are $\frac{1}{8}$, $\frac{1}{40}$, and $\frac{1}{160}$, while the first 3
  moments of $\mu_2 (\XX'_3 (0) )$ are $\frac{1}{2}$, $\frac{51}{160}$, and $\frac{73}{320}$. Finally,
  $\Exp [  \mu_2 (\XX'_3 (0) ) ( D_2 (\XX'_3(0)) )^2 ] = \frac{1}{80}$.
\end{lm}
\begin{proof}
The $3$ points of $\XX_3(0)$   induce  a partition of the interval $[0,1]$
into uniform spacings $S_{1}, S_{2}, S_{3}, S_{4}$, enumerated left to right (for this proof we suppress the first index in the  notation above).
For ease of notation, write $D:=  D_2 (\XX'_3(0))$ and $\mu := \mu_2 (\XX'_3 (0) )$ for the duration of this proof.
 Then $D = \min \{ S_2, S_3 \} \eqd S_1 / 2$, by (\ref{min1}).
 Moreover, $\min \{S_2, S_3\}$  is equally likely to be either $S_2$ or $S_3$. In the former case,
$\mu = S_1 + \frac{1}{2} \min \{S_2, S_3\}$, while in the latter case $\mu = 1- S_4 - \frac{1}{2} \min \{S_2, S_3\}$.
Using (\ref{min2}), we obtain the following characterization of the joint distribution of $\mu$ and $D$.
\begin{equation}
\label{mud}
(\mu, D) \eqd \begin{cases}
(S_1 + \frac{1}{4} S_2 , \frac{1}{2} S_2 ) & \textrm{ with probability } \frac{1}{2} \\
(1 - S_1 - \frac{1}{4} S_2 , \frac{1}{2} S_2 ) & \textrm{ with probability } \frac{1}{2} .\end{cases}
\end{equation}
In particular,
$\Exp [ D^k ] = 2^{-k} \Exp [ S_{1}^k ]$, which gives (\ref{momd})
by the $n=3$, $\alpha =k$, $\beta= 0$ case of (\ref{multispace}).

For the moments of $\mu$,  we have from (\ref{mud}) that
$\mu$ has the  distribution of $W: = S_1 + \frac{1}{4} S_2$ with probability $1/2$
or $1- W$ with probability $1/2$. So we have
\begin{align*}
\Exp [ \mu^ k ] & = \frac{1}{2} \Exp [ W^k ] + \frac{1}{2} \Exp [ (1-W)^k ]
 = \frac{1}{2} w_k + \frac{1}{2} \sum_{j=0}^{ k } {\binom k  j} (-1)^j w_j ,\end{align*}
 where $w_k := \Exp [ W^k]$. Since $w_k = \Exp [ (S_1 + \frac{1}{4} S_2)^k ]$, we compute
 \[ w_k = \sum_{j=0}^k {\binom k  j} 4^{-j}   \Exp [ S_1^{k-j} S_2 ^{j }  ] = 6 \frac{k!}{(k+3)!} \sum_{j=0}^k   4^{-j},\]
 by the $n=3$, $\alpha = k-j$,  $\beta = j$,  case of (\ref{multispace}). Thus we obtain
 \[ w_k = \frac{8 (1-4^{-(k+1)})}{(k+1)(k+2)(k+3) } .\]
It follows that
\[ \Exp [ \mu^k ] =  \frac{1}{2} w_k + 4 \sum_{j=0}^k \frac{k!}{(j+3)! (k-j)!} (-1)^j - \sum_{j=0}^k \frac{k!}{(j+3)! (k-j)!} ( - 1/4)^j .\]
We deduce (\ref{mommu}), after simplification, from the claim that, for any $z \in \R$,
\begin{align}
\label{3sum}
S(z) & := \sum_{j=0}^k \frac{k!}{(j+3)!(k-j)!} (-z)^j \nonumber\\
& = \frac{k!}{2 z^3 (k+3)!} \left[ z^2 (k+2) (k+3) + 2 -2z (k+3) - 2 (1-z)^{k+3} \right] .\end{align}
Thus it remains to verify (\ref{3sum}). To this end, note that
\begin{align*}
S(z) & = \frac{k!}{(k+3)!} \sum_{j=0}^k {\binom {k+3} {j+3}} (-z)^ j \\
& = \frac{k!}{(k+3)!} \left[ -z^{-3} \sum_{j=0}^{k+3} {\binom {k+3} {j} } (-z)^j + z^{-1} {\binom {k+3} 2} - z^{-2} {\binom {k+3}  1} + z^{-3} {\binom {k+3} 0 } \right] \\
& = \frac{k!}{z^3 (k+3)!} \left[ -(1-z)^{k+3} + \frac{1}{2} z^2 (k+2) (k+3) - z (k+3) + 1 \right] ,\end{align*}
which gives the claim (\ref{3sum}).

For the final statement in the lemma, we have from (\ref{mud}) that
\[ \Exp [ \mu D^2] = \frac{1}{2} \Exp [ (S_1 + \tfrac{1}{4} S_2 ) (\tfrac{1}{4} S_2^2 ) ] +   \frac{1}{2} \Exp [ (1 - S_1 - \tfrac{1}{4} S_2 ) (\tfrac{1}{4} S_2^2 ) ]
= \frac{1}{8} \Exp [ S_2^2 ] = \frac{1}{80} ,\]
by (\ref{multispace}).
\end{proof}

We can also obtain explicit expressions for the densities of $D$ and $\mu$. Since $D \eqd S_1/2$, the density of $D$ is
$f_D(r) = 3 (1-2r)^2$ for $r \in [0,1/2]$. In addition, $\mu$ has density $f_\mu$ given by
\begin{equation}
\label{mudensity}
 f_\mu (r ) = \begin{cases}
4 r [ 3 (1-r) - 4r ] & \textrm{ if } r \in [0,1/4] \\
2 - 4 r (1-r) & \textrm{ if } r \in [1/4,3/4] \\
4 (1-r) [ 3 r - 4(1-r) ] & \textrm{ if } r \in [3/4,1]
\end{cases}.\end{equation}
Indeed, with the representation of $\mu$ as either $W$ or $1-W$ with probability $1/2$ of each, we have
$\Pr [ \mu \leq r ] = \frac{1}{2} \Pr [ W \leq r ] + \frac{1}{2} ( 1 - \Pr [ W < 1-r ] )$.
Assuming that $W$ has a density $f_W$ (which indeed it has, as we will show below), we get
\begin{equation}
\label{den1}
 f_\mu (r) = \frac{1}{2} f_W (r) + \frac{1}{2} f_W (1 -r) .\end{equation}
Using the fact that $W \eqd S_1 + \frac{1}{4}S_2$, we can use the joint distribution of $(S_1, S_2)$ to calculate
\begin{align*} \Pr [ W \leq r ] & = \int_0^1 \ud x_1 \int_0^{1-x_1} \ud x_2 6 (1-x_1 -x_2 ) \1 \{ x_1 + \frac{1}{4} x_2 \leq r \} \\
& = \int_0^r \ud x_1   \int_0^{(4(r-x_1)) \wedge (1-x_1)} \ud x_2 6 (1-x_1 -x_2 ) .\end{align*}
After some routine calculation, we then obtain
\[ \Pr [ W \leq r] = \begin{cases}
1  - \frac{4}{3} (1-r)^3 + \frac{1}{3} (1-4r)^3 & \textrm{ if } r \in [0,1/4] \\
1 - \frac{4}{3} (1-r)^3    & \textrm{ if } r \in [1/4,1] .\end{cases} \]
Hence $W$ has density $f_W$ given by
\[ f_W (r) = \begin{cases}
4 (1-r)^2 - 4 (1-4r)^2 & \textrm{ if } r \in [0,1/4] \\
4 (1-r)^2   & \textrm{ if } r \in [1/4,1] .\end{cases} \]
Then (\ref{mudensity}) follows from (\ref{den1}).

\section{Appendix 2: Continuity of random variables}
\label{appendix2}

In this appendix we give some results that will allow us to deduce
the absolute continuity of certain distributions specified as solutions to fixed-point equations:
specifically, we use these results in the proof of Proposition \ref{Lcontinuous}.
The results in this section may well be known, but we were unable to find 
a reference for them in a form directly suitable for our application, and so we include the (short) proofs.

\begin{lm}\label{lemconts1}
Let $X$ and $Y$ be  independent random variables such that $X$ has an absolutely
continuous distribution. Then for any  $a\ne 0$ we have $\Pr [ XY=a ] =0$.
Morevoer, if $\Pr[ Y=0 ] =0$, then $XY$ is an absolutely continuous random variable.
\end{lm}
\begin{proof}
For the moment assume that $\Pr [X<0]$, $\Pr[X>0]$, $\Pr [Y<0]$, and $\Pr[Y>0]$ are all positive.
Take some $0 < c< d$. Then
\begin{align*}
\Pr [ XY\in (c,d) ] & = \Pr [ \log(X)+\log(Y)\in (\log c,\log d) \mid X>0, \,  Y>0 ]
\Pr [ X>0] \Pr [ Y>0]
 \\
& ~~{} + 
\Pr [ \log(-X)+\log(-Y) \in (\log c,\log d) \mid X<0, \,  Y<0 ]
\Pr [ X<0] \Pr [ Y<0] .
\end{align*}
Note that conditioning $X$ on the event $X>0$ (or $X<0$) preserves the continuity of $X$ and the independence
of $X$ and $Y$. Then since the sum of two independent random variables at least one of which absolutely continuous is also absolutely continuous (see \cite[Theorem 5.9, p.\ 230]{moran}) we have
\[ 
\Pr [ \log(X)+\log(Y)\in (\log c,\log d) \mid X>0, \,  Y>0 ]
 = \int_{\log c}^{\log_d} f_+(x) \ud x,
\]
and
\[
\Pr [ \log(-X)+\log(-Y)\in (\log c,\log d) \mid X<0, \, Y<0 ]
=\int_{\log c}^{\log_d} f_-(x) \ud x
\]
for  suitable probability densities $f_+$ and $f_-$.
After the substitution $u = \re^x$, this yields
\begin{align*}
\Pr[ XY\in (c,d) ] &= \int_{c}^{d} 
 \frac{\Pr [ X>0] \Pr [Y>0 ]  f_+(\log u)+\Pr [X<0 ]\Pr [Y<0 ]f_-(\log u)}u \ud u.
\end{align*}
This expression is also valid if some of the probabilities
for $X$ and $Y$ in the numerator of the integrand are zero. Therefore, we have, for any $0 < c < d$, 
\begin{align}
\label{fu}
\Pr [ XY\in (c,d) ] &= \int_{c}^{d}  f(u) \ud u,
\end{align}
for some function $f(u)$ defined for $u>0$.
A similar argument applies to the case $c<d < 0$;
then (\ref{fu}) is valid for any $c < d <0$ as well,
extending $f(u)$ for strictly negative $u$.
In particular, it follows that $\Pr [ XY=a ] =0$ for $a\ne 0$.

Now if $\Pr [Y=0 ]=0$, then $\Pr [ XY \neq 0 ] =1$. Then  we can set $f(0)=0$ 
so that (\ref{fu}) holds for \emph{all} $c,d \in \R$.
\end{proof}

\begin{lm}\label{lemconts2}
Suppose that a random variable $L$ satisfies the distributional equation
\begin{equation*}
 L  \eqd 
\begin{cases}
 Z_1 & \textrm{with~probability~} p_1 \\
 \vdots & \\
Z_n & \textrm{with~probability~} p_n  ,
 \end{cases} 
 \end{equation*}
where $n\in \N$, $\sum_{i=1}^n p_i=1$, $p_i>0$, and
each $Z_i$ is 
an absolutely continuous random variable.
Then $L$ is absolutely continuous.
\end{lm}
\begin{proof}
Suppose $Z_i$ has a density $f_i$. Then for any 
$-\infty\le a< b\le +\infty$ we have
\begin{align*}
\Pr [ L\in (a,b) ] 
=\sum_{i=1}^n p_i \Pr[ Z_i \in (a,b) ] 
=\sum_{i=1}^n p_i \int_a^b f_i(x) \ud x
=\int_a^b  \left[\sum_{i=1}^n p_i f_i(x) \right] \ud x,
\end{align*}
which yields the statement of lemma.
\end{proof}

\section*{Acknowledgements}

Parts of this work were done at the University of Strathclyde,
where the third author was also employed, during a couple of visits
by the second author, who is grateful for the hospitatlity of that institution.


\begin{thebibliography}{99}

\bibitem{bm} C. Benassi and F. Malagoli,
The sum of squared distances under a diameter constraint, in arbitrary dimension,
{\em Arch. Math.} {\bf 90} (2008) 471--480.

\bibitem{degr} E. De Giorgi and S. Reimann,
The $\alpha$-beauty contest: Choosing numbers, thinking
intervals, {\em Games Econom. Behav.} {\bf 64} (2008)  470--486.

\bibitem{erdos} P. Erd\H os, On the smoothness properties of a family of Bernoulli convolutions,
{\em Amer. J. Math.} {\bf 62} (1940) 180--186.

\bibitem{gkw}
M. Grinfeld, P.A. Knight, and A.R. Wade,
Rank-driven Markov processes,
{\em J. Stat. Phys.} {\bf 146} (2012) 378--407.

 \bibitem{hughes} B.D. Hughes, {\em Random Walks and Random
 Environments; Volume 1: Random Walks}, Clarendon Press,
 Oxford, 1995.

  \bibitem{jk} N.L. Johnson and S. Kotz,
 Use of moments in studies of limit distributions arising from iterated random subdivisions of an interval,
 {\em Statist. Probab. Lett.} {\bf 24} (1995) 111--119.

 \bibitem{keynes}
J.M. Keynes,
{\em The General Theory of Employment, Interest and Money},
 Macmillan, London,
1936.

\bibitem{kr}
P.L. Krapivsky and S. Redner,
Random walk with shrinking steps,
{\em Amer. J. Phys.} {\bf 72} (2004) 591--598.

\bibitem{moran} P.A.P. Moran, {\em An Introduction to Probability Theory}, Clarendon Press, Oxford,
1968 (paperback ed., with corrections, 2002).

\bibitem{moulin} H. Moulin, {\em Game Theory for the Social Sciences}, 2nd ed., New York University Press, New York, 1986.

\bibitem{pemantle}
R. Pemantle,  A survey of random processes with reinforcement,
{\em Probab. Surv.} {\bf 4} (2007) 1--79.

\bibitem{pw1} M.D. Penrose and A.R. Wade,
Random minimal directed spanning trees and Dickman-type distributions,
{\em Adv. in Appl. Probab.} {\bf 36} (2004) 691--714.

\bibitem{pwong}
M.D. Penrose and A.R. Wade,
Limit theory for the random on-line nearest-neighbor graph,
{\em Random Structures Algorithms}
{\bf 32}   (2008) 125--156.

 \bibitem{pill} F. Pillichshammer, On the sum of squared distances in the Euclidean plane,
 {\em Arch. Math.} {\bf 74}  (2000) 472--480.

\bibitem{wit}  H.S. Witsenhausen,
On the maximum of the sum of squared distances under a diameter constraint,
{\em Amer. Math. Monthly} {\bf 81} (1974) 1100--1101.

\end{thebibliography}
\end{document}